\documentclass[letterpaper, 11pt]{article}
\usepackage[utf8]{inputenc}

\usepackage{geometry}
\usepackage{amsmath}
\usepackage{amsthm}
\usepackage{amssymb}
\usepackage{graphicx,enumitem,mdframed,mathrsfs}

\usepackage{cite}
\usepackage{comment}
\usepackage{subcaption}
\usepackage{epsfig} 

\usepackage{tikz}
\usetikzlibrary{shapes.geometric, arrows, backgrounds, scopes, positioning}

\usepackage{algorithm}
\usepackage{algpseudocode}
\usetikzlibrary{fit}
\usepackage{xcolor}
\usepackage{multirow}
\usepackage{hyperref}
\usepackage{authblk}

\providecommand{\norm}[1]{\ensuremath{\left\lVert#1\right\rVert }}
\providecommand{\mnorm}[1]{\ensuremath{\left\lvert#1\right\rvert}}
\providecommand{\floor}[1]{\ensuremath{\left\lfloor#1\right\rfloor}}

\newtheorem{theorem}{\bfseries Theorem}
\newtheorem{lemma}{\bfseries Lemma}

\newtheorem{remarks}{\bfseries Remarks}
\newtheorem{corollary}{\bfseries Corollary}
\newtheorem{fact}{\bfseries Fact}

\def\R{\mathbb{R}}

\def\A{A^TA}
\def\N{\mathcal{N}}

\DeclareRobustCommand{\bigO}{%
  \text{\usefont{OMS}{cmsy}{m}{n}O}%
}

\title{
Iterative Pre-Conditioning for Expediting the Gradient-Descent Method:\\ The Distributed Linear Least-Squares Problem
}

\author{Kushal Chakrabarti$^\star$, Nirupam Gupta$^\dagger$, and Nikhil Chopra$^\star$
\thanks{$^\star$ University of Maryland, College Park, Maryland 20742, U.S.A. \\
$^\dagger$ \'Ecole polytechnique f\'ed\'erale de Lausanne (EPFL) CH-1015 Lausanne \\
Emails: {\em kchakrabarti0@gmail.com}, {\em nirupam115@gmail.com} and {\em nchopra@umd.edu}}%
}

\date{}

\begin{document}

\maketitle

\begin{abstract}


This paper considers the multi-agent linear least-squares problem in a {\em server-agent} network. In this problem, the system comprises multiple agents, each having a set of local data points, that are connected to a server. The goal for the agents is to compute a linear mathematical model that optimally fits the collective data points held by all the agents, without sharing their individual local data points. This goal can be achieved, in principle, using the server-agent variant of the traditional iterative gradient-descent method. The gradient-descent method converges {\em linearly} to a solution, and its {\em rate of convergence} is lower bounded by the {\em conditioning} of the agents' collective data points. If the data points are {\em ill-conditioned}, the gradient-descent method may require a large number of iterations to converge.\\

We propose an {\em iterative pre-conditioning} technique that mitigates the deleterious effect of the conditioning of data points on the rate of convergence of the gradient-descent method. 
We rigorously show that the resulting {\em pre-conditioned} gradient-descent method, with the proposed iterative pre-conditioning, achieves {\em superlinear} convergence when the least-squares problem has a unique solution. In general, the convergence is {\em linear} with improved {\em rate of convergence} in comparison to the traditional gradient-descent method and the state-of-the-art {\em accelerated} gradient-descent methods. We further illustrate the improved rate of convergence of our proposed algorithm through experiments on different real-world least-squares problems in both noise-free and noisy computation environment.



\end{abstract}
\newpage

\tableofcontents
\newpage

\section{Introduction}
\label{sec:intro}
In this paper, we consider the multi-agent distributed linear least-squares problem. The nomenclature {\em distributed} here refers to the data points being distributed across multiple agents. Specifically, we consider a system that comprises of multiple agents where each agent has a set of local data points. The agents can communicate bidirectionally with a central server as shown in Fig.~\ref{fig:sys}. However, there is no inter-agent communication, and the agents cannot share their individual local data points with the server. The goal for the agents is to compute a linear mathematical model that optimally fits the collective data points of all the agents. For doing so, as a single agent does not have access to all the data points, the agents must collaborate with the server. Throughout this paper, we refer to the above described system architecture as {\em server-agent network}, and we assume the system to be synchronous unless mentioned otherwise. \\

\tikzstyle{master} = [rectangle, rounded corners, minimum width=1.7cm, minimum height=1cm,text centered, text width=1cm, draw=black, fill=blue!30]
\tikzstyle{machine} = [rectangle, minimum width=1cm, minimum height=1cm,text centered, text width=2cm, draw=black, fill=blue!10]
\tikzstyle{dots} = [circle, inner sep=0pt,minimum size=2pt, draw=black, fill=blue!50!cyan]
\tikzstyle{arrow} = [thick,<->,>=stealth]

\begin{figure}[thpb]  
\centering
\begin{tikzpicture}[node distance = 1.5cm, auto]
    \node (master) [master] {Server};
    \node (m/c2) [machine, below  = of master] {Agent 2\\ $(A^2,B^2)$};
    \node (m/c1) [machine, left = of m/c2, xshift=1cm] {Agent 1\\ $(A^1,B^1)$};
    \node (d1) [dots, right = of m/c2, xshift=-0.8cm] {};
    \node (d2) [dots, right = of d1, xshift=-1.4cm] {};
    \node (d3) [dots, right = of d2, xshift=-1.4cm] {};
    \node (m/c3) [machine, right = of d3, xshift=-0.8cm] {Agent m\\ $(A^m,B^m)$};
    
    \draw[arrow] (master) -- (m/c1);
    \draw[arrow] (master) -- (m/c2);
    \draw[arrow] (master) -- (m/c3);
\end{tikzpicture}
\caption{System architecture.}
\label{fig:sys}
\end{figure}
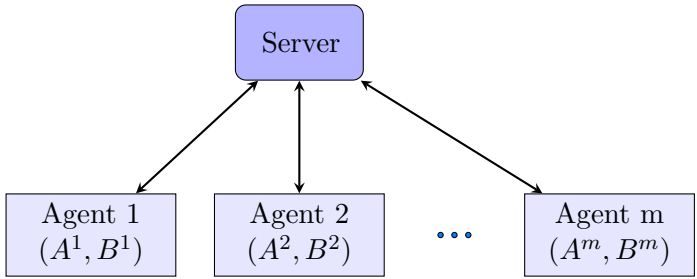

Specifically, we consider a system with $m$ agents. Each agent $i$ has a set of $n_i$ data points represented by the rows of a $(n_i \times d)$-dimensional real-valued matrix $A^i$, and the elements of a $n_i$-dimensional real-valued vector $B^i$. That is, for each agent $i$, $A^i \in \R^{n_i \times d}$ and $B^i \in \R^{n_i}$. The goal for the agents is to compute a parameter vector $x^* \in \R^d$ such that
\begin{align}
    x^* \in X^* = \arg \min_{x \in \R^d} \sum_{i = 1}^m \frac{1}{2}\norm{A^i x - B^i}^2. \label{eqn:opt_1}
\end{align}
We refer to matrix $A^i$ and vector $B^i$ as local {\em data matrix} and local {\em observations}, respectively of agent $i$. For each agent $i$, we define a local {\em cost function}
\begin{align}
    F^i(x) =   \frac{1}{2}\norm{A^i x - B^i}^2, \quad \forall x \in \R^d.  \label{eqn:dist_opt}
\end{align}
It is easy to see that solving for the optimization problem~\eqref{eqn:opt_1} is equivalent to computing a minimum point of the {\em aggregate cost function} $\sum_{i = 1}^m F^i(x)$. \\

Common applications of the above linear least-squares problem include linear regression, state estimation, and hypothesis testing~\cite{zhang2017stochastic, zhu2018linear}. Also, a wide range of supervised machine learning problems can be modelled as a linear least-squares problem, such as the supply chain demand forecasting~\cite{carbonneau2008application}, prediction of online user input actions~\cite{canny2013method}, and the problem of selecting sparse linear solvers~\cite{bhowmick2006application}. In several contemporary applications, the data points exist as dispersed over several sources. Due to industry competition, administrative regulations, and user privacy, it is almost impossible to integrate the data points from those isolated sources~\cite{yang2019federated}. This has brought the researcher community's focus towards collaboratively fitting a prediction model such as~\eqref{eqn:opt_1} while keeping all the raw data in its device, without requiring data-transaction among the sources and to the server~\cite{smith2017federated,yang2019federated}. Herein lies our motivation to improve upon the state-of-the-art method for solving~\eqref{eqn:opt_1} distributively in a server-agent network.\\

As elaborated below, the agents can solve for an optimal linear model~\eqref{eqn:opt_1} using the server-agent network version of the traditional gradient-descent method~\cite{bertsekas1989parallel}.

\subsection{Background: Gradient-Descent Method}
\label{sub:dgd}
The gradient-descent method is an iterative algorithm wherein the server maintains an estimate of a solution defined by~\eqref{eqn:opt_1} and updates it iteratively using gradients of agents' local cost functions. To be precise, for each iteration $t = 0, \, 1, \ldots$ , let $x(t) \in \R^d$ denote the estimate maintained by the server. The initial estimate $x(0)$ may be chosen arbitrarily from $\R^d$. For each iteration $t$, the server broadcasts $x(t)$ to all the agents. Each agent $i$ computes the gradient of its local cost function $F^i(x)$ at $x = x(t)$ denoted by $g^i(t)$. Specifically, 
\begin{align}
    g^i(t) = \nabla F^i(x(t)) = \left( A^i \right)^T \, \left(A^i \, x(t) - B^i\right), \quad \forall i \in \{1, \ldots, \, m\}, ~ t \in \{0, \, 1, \ldots \}, \label{eqn:g_i}
\end{align}
where $(\cdot)^T$ denotes the transpose. The agents send their computed gradients $\{g^i(t), ~ i = 1, \ldots, \, m \}$ to the server. Upon receiving the gradients, the server updates $x(t)$ as follows:
\begin{align}
    x(t+1) = x(t) - \delta \, \sum_{i = 1}^m g^i(t), \quad \forall t \in \{0, \, 1, \ldots \}, \label{eqn:dgd}
\end{align}
where $\delta$ is a positive scalar real value commonly referred as the {\em step-size}. 
Let $g(t)$ denote the sum of all the agents' gradients, that is, for all $t$,
\begin{align}
    g(t) = \sum_{i = 1}^m g^i(t). \label{eqn:sum_grads}
\end{align}
Substituting from~\eqref{eqn:sum_grads} in~\eqref{eqn:dgd}, we can see that the gradient-descent method in a server-agent network (ref. Fig.~\ref{fig:sys}) is equivalent to its centralized version where the cost function is equal to the summation of all the agents' local cost functions $\sum_{i = 1}^m F^i(x)$ (see~\cite{bertsekas1989parallel}).
Therefore, for small enough step-size $\delta$, the sequence of gradients $\{g(t), ~ t  = 0, \, 1, \ldots\}$ converges {\em linearly} to $0_d$. To be precise, for sufficiently small $\delta$ there exists $\mu \in [0, \, 1)$ such that~\cite{fessler2008image},
\begin{align*}
    \norm{g(t)} \leq \mu^{t} \norm{g(0)}, \quad \forall t \in \{0, \, 1, \ldots\}.
\end{align*}
Equivalently, due to convexity of the optimization problem~\cite{bertsekas1989parallel}, the sequence of estimates $\{x(t), ~ t  = 0, \, 1, \ldots \}$ also converge linearly to a point in the solution set $X^*$ defined in~\eqref{eqn:opt_1}. The scalar $\mu$ is referred as the {\em rate of convergence}~\cite{kelley1999iterative}. As is evident from above, a smaller value of $\mu$ implies a faster convergence, and vice-versa. However, as elaborated later in Section~\ref{sec:comp}, the value of $\mu$ is lower bounded by a non-negative value $\mu_{GD}$ that depends upon the {\em condition number} of the data matrix
\begin{align}
    A = \begin{bmatrix} (A^1)^T, \ldots, \, (A^m)^T \end{bmatrix}^T. \label{eqn:data_matrix}
\end{align}
Note that the matrix $A$ is of dimension $(\sum_{i = 1}^m n_i) \times d$. \\

We propose an {\em iterative pre-conditioning} technique that improves upon the rate of convergence of the gradient-descent method in a server-agent network. Specifically, in each iteration, the server multiplies the aggregate of the agents' gradients $g(t)$ by a {\em pre-conditioner matrix} $K$ before updating the local estimates. However, unlike the classical pre-conditioning techniques~\cite{fessler2008image}, in our case, the server iteratively updates the pre-conditioner matrix $K$. Hence, the name {\em iterative pre-conditioning}. A detailed description of the resulting pre-conditioned gradient-descent method and its convergence properties are given in Section~\ref{sec:algo}.\\

Before we present our proposed technique, let us review below the existing state-of-the-art techniques for improving the rate of convergence of the traditional gradient-descent method. As elaborated later in Section~\ref{sec:comp}, the techniques disucssed below are applicable to the server-agent network.

\subsection{Related Work} 
\label{sec:lit_survey}


In the seminal work~\cite{nesterov27method}, Nesterov showed that the use of {\em momentum} can significantly {\em accelerate} the gradient-descent method. Recently, there has been work on the applicability of Nesterov's {\em accelerated gradient-descent} method to the server-agent network, such as~\cite{azizan2019distributed} and references therein. Azizan-Ruhi et al.~\cite{azizan2019distributed} have proposed an {\em accelerated projection} method, which is a combination of the Nesterov's accelerated gradient-descent method with a projection operation. Azizan-Ruhi et al. have shown through experiments that their {\em accelerated projection} method converges faster compared to the variants of the Nesterov's accelerated gradient-descent method and the {\em heavy-ball method}~\cite{polyak1964some}. However, they do not provide any theoretical guarantee for the improvement in the convergence speed. Also, Azizan-Ruhi et al. only consider a degenerate case of the optimization problem~\eqref{eqn:opt_1} where the set of linear equations $A^i x = B^i, ~ i = 1, \ldots, \, m$, has a unique solution. We consider a more general setting wherein the minimum value of the aggregate cost function $\sum_{i = 1}^m F^i(x)$ need not be zero. Also, in general, the solution for the optimization problem~\eqref{eqn:opt_1} need not be unique.\\

The {\em heavy-ball method}~\cite{polyak1964some} is another momentum-based accelerated variant of the gradient-descent method. In contrast to Nesterov's method, which uses the current and the previous momentum terms, the heavy-ball method only uses the current momentum term for updating the current estimate. The heavy-ball method is guaranteed to converge faster than both the gradient-descent method and Nesterov's accelerated method. For the case when the optimization problem~\eqref{eqn:opt_1} has a unique solution, both these accelerated methods, namely the heavy-ball method and Nesterov's accelerated gradient-descent method, are known to converge linearly with rate of convergence smaller than the above traditional gradient-descent method~\cite{lessard2016analysis,fazlyab2018analysis}.\\

The second-order {\em Newton's method} has a {\em quadratic convergence}, and therefore, it has a {\em superlinear} rate of convergence~\cite{kelley1999iterative}. However, Newton's method cannot be implemented in the distributed server-agent network unless the agents share their local data points with the server. Quasi-Newton methods, on the other hand, such as {\em BFGS}~\cite{kelley1999iterative} can be executed in the server-agent network architecture.\footnote{BFGS stands for Broyden, Fletcher, Goldfarb, and Shanno, who proposed the algorithm~\cite{kelley1999iterative}.} However, similar to Newton's method, BFGS also needs the solution of the optimization problem~\eqref{eqn:opt_1} to be unique.


\subsection{Summary of Our Contributions}
\label{sub:contri}
We propose an {\em iterative pre-conditioning} technique for improving the rate of convergence of the traditional gradient-descent method, when solving the aforementioned distributed linear least-squares problem in a server-agent network. Details of our algorithm are presented in Section~\ref{sec:algo}. We summarize below our key contributions.

\begin{enumerate}
    \item We show, in Sections~\ref{sub:conv}, that in general our algorithm converges {\em linearly} to a solution defined by~\eqref{eqn:opt_1} with improved rate of convergence in comparison to the traditional gradient-descent method described above in Section~\ref{sub:dgd} for the server-agent network. Refer Section~\ref{sec:comp} for a rigorous comparison between our algorithm and the traditional gradient-descent method.
    
    \item For the special case when the solution of the least-squares problem~\eqref{eqn:opt_1} is unique we show, in Section~\ref{sub:conv}, that our algorithm converges {\em superlinearly}. This is an improvement over the server-agent network versions of the heavy-ball method, Nesterov's accelerated gradient-descent method, and the accelerated projection method  which are only known to converge {\em linearly}~\cite{lessard2016analysis,fazlyab2018analysis,azizan2019distributed}. See Section~\ref{sec:comp} for more details.
    
    \item We show, in Section~\ref{sec:quad}, that the proposed algorithm is also applicable to the more general distributed convex quadratic minimization problem in a server-agent network.
    
    \item We present an analysis of our proposed algorithm regarding its sensitivity towards system noise. See Appendix~\ref{sec:noise} for details.
    
    \item We illustrate our obtained theoretical comparisons with existing algorithms through numerical experiments on different real-world datasets. Detailed presentation of our experimental results is given in Section~\ref{sec:exp}. These results also suggest that our proposed algorithm is less sensitive to system noise than the aforementioned existing methods.

\end{enumerate}




The idea of iterative pre-conditioning was first proposed in our conference paper~\cite{chak2020ipc}. However, in~\cite{chak2020ipc}, we consider a special case when the solution set $X^*$ (defined in~\eqref{eqn:opt_1}) is singleton. In comparison to~\cite{chak2020ipc}, this paper includes a more detailed convergence analysis of our proposed algorithm for the more general least-squares problem whose solution~\eqref{eqn:opt_1} may not be unique. The current paper presents rigorous comparisons between the convergence of our algorithm and other state-of-the-art algorithms. The current paper also includes theoretical and experimental evaluations of our algorithm in presence of system noise, applicability to solving of the more general convex quadratic problem in the server-agent network. The experiments in the current paper are also extensive compared to~\cite{chak2020ipc}.


\subsection*{Paper Outline}

The rest of this paper is organized as follows. In Section~\ref{sec:algo}, we present our proposed algorithm and its convergence properties. 
Section~\ref{sec:comp} presents rigorous comparisons between the convergence rate of our proposed algorithm and other state-of-the-art algorithms. Section~\ref{sec:quad} presents the extended applicability of the proposed algorithm to solving of convex quadratic minimization problem in a server-agent network.
Section~\ref{sec:analysis} presents the formal proof for the main convergence result of our algorithm. Section~\ref{sec:exp} presents experimental evaluations. Finally, the contributions made in the paper are summarized in Section~\ref{sec:summary}. The paper comprises two appendices. In Appendix~\ref{sec:proof}, we present formal proofs of some elementary results that are used in the proof of the main result presented in Section~\ref{sec:analysis}. In Appendix~\ref{sec:noise}, we discuss in detail the effects of system noise on the proposed algorithm.


\newpage

\section{Proposed Algorithm}
\label{sec:algo}
This section presents our algorithm, its computational complexity, and its formal convergence properties. \\
 

The proposed algorithm is built on top of the gradient-descent method in a server-agent network described in Section~\ref{sub:dgd}. However, as elaborated below, before updating the current estimates using the aggregate of the agents' gradients, the server multiplies the gradients by a matrix. This technique of multiplication of the gradients by a matrix, in the gradient-descent method, is popularly known as {\em pre-conditioning}~\cite{nocedal2006numerical}. The matrix being multiplied is known as the {\em pre-conditioner} matrix. Unlike existing pre-conditioning techniques~\cite{fessler2008image}, in our case, the {\em pre-conditioner} matrix gets updated after each iteration. Hence, we name our pre-conditioning technique as {\em iterative pre-conditioning}.\\


In each iteration $t \in \{0, \, 1, \ldots\}$, the server maintains an estimate $x(t)$ of a minimum point~\eqref{eqn:opt_1}, and a $(d \times d)$-dimensional real-valued pre-conditioner matrix $K(t)$. The initial estimate $x(0)$ and matrix $K(0)$ are chosen arbitrarily from $\R^d$ and the set of $(d \times d)$-dimensional real-valued matrices $\R^{d \times d}$, respectively. Also, before initializing the iterative process, the server chooses three non-negative scalar real-valued parameters $\alpha$, $\delta$ and $\beta$. The parameter $\beta$ is sent to the agents.\\

Recall, from Section~\ref{sec:intro}, that each agent $i \in \{1, \ldots, \, m\}$ has a local cost function 
\[F^i(x) = \frac{1}{2} \norm{A^i x - B^i}^2, \quad x \in \R^d,\]
where the pair $(A^i, \, B^i)$ denotes the local data points held by agent $i$. For each iteration $t \geq 0$, the algorithm comprises four steps presented below.

\subsection{Steps for Each Iteration}
\label{sec:algo_steps}
The algorithm comprises of four steps described below. The steps are executed collaboratively by the server and the agents. This algorithm has been presented in our previous work~\cite{chak2020ipc}.

\begin{itemize}
\setlength\itemsep{0.5em}
    \item {\em Step 1:} The server sends the estimate $x(t)$ and the matrix $K(t)$ to each agent $i$.
    \item {\em Step 2:} Each agent $i$ computes the gradient
    \begin{align}
        g^i(t) = \nabla F^i(x(t)) = \left( A^i \right)^T \, \left(A^i \, x(t) - B^i\right). \label{eqn:g_i}
    \end{align}
    Let $I$ be the $(d \times d)$-dimensional identity matrix. Let $e_j$ and $k_j(t)$ denote the $j$-th columns of matrices $I$ and $K(t)$, respectively. In the same step, each agent $i$ computes a set of vectors $\left\{R^i_j(t): ~ j = 1, \ldots, \, d \right\}$ such that for each $j$,
    \begin{align}\label{eqn:Rij}
        R^i_j(t) \hspace{-0.2em} & = \hspace{-0.2em}\left((A^i)^T \hspace{-0.2em}A^i + \dfrac{\beta}{m}I\right) k_j(t) \hspace{-0.1em} -  \dfrac{1}{m} e_j,
    \end{align}
    where $\beta$ is a non-negative real value. 
    \item {\em Step 3:} Each agent $i$ sends the gradient $g^i(t)$ and the set $\left\{R^i_j(t), ~ j = 1, \ldots, \, d \right\}$ to the server.
    \item {\em Step 4:} The server updates the matrix $K(t)$ to $K(t+1)$ such that
    \begin{align}
        k_j(t + 1) = k_j(t) - \alpha \sum_{i=1}^m R^i_j(t), \quad j = 1,...,d, \label{eqn:kcol_update}
    \end{align}
    where $\alpha$ is a positive constant real value. Then, the server updates the estimate $x(t)$ to $x(t + 1)$ such that
    \begin{align}
        x(t+1) = x(t) - \delta K(t + 1) \sum_{i=1}^m g^i(t), \label{eqn:x_update}
    \end{align}
    where $\delta$ is a positive constant real value, called the {\em step-size}.
\end{itemize}


\begin{algorithm}
  \caption{Distributed {\em iterative pre-conditioning} for the gradient-descent method.}\label{algo_1}
  \begin{algorithmic}[1]
    \State The server initializes $x(0) \in \R^d$, $K(0) \in \R^{d \times d}$, $\alpha > 0$, $\delta > 0$ and $\beta \geq 0$.
    \For{\texttt{$t=0, \, 1, \, 2, \ldots$}}
      \State The server sends $x(t)$ and $K(t)$ to each agent $i \in \{1, \ldots, \, m\}$.
      \State Each agent $i$ computes the gradient $g^i(t)$ as defined by~\eqref{eqn:g_i}, and a set of vectors $\left\{R^i_j(t), ~ j = 1, \ldots, \, d \right\}$ as defined by~\eqref{eqn:Rij}.
      \State Each agent $i$ sends $g^i(t)$ and the set $\left\{R^i_j(t), ~ j = 1, \ldots, \, d \right\}$ to the server.
      \State The server updates $K(t)$ to $K(t + 1)$ as defined by~\eqref{eqn:kcol_update}.
      \State The server updates the estimate $x(t)$ to $x(t + 1)$ as defined by~\eqref{eqn:x_update}.
    \EndFor
  \end{algorithmic}
\end{algorithm}

The algorithm is summarized in Algorithm~\ref{algo_1}. Next, we discuss the computational complexity of the algorithm. 

\subsection{Computational Complexity}
We present the computational complexity of Algorithm~\ref{algo_1}, for both the agents and the server, in terms of the total number of floating-point operations (\textit{flops}) required per iteration. As floating-point multiplication is significantly costlier than floating-point additions~\cite{golub2012matrix}, we ignore the additions while counting the total number of flops.\\

For each iteration $t$, each agent $i$ computes the gradient $g^i(t)$, defined in~\eqref{eqn:g_i}, and $d$ vectors $\{\R^i_j(t) : ~ j=1,...,d\}$, defined in~\eqref{eqn:Rij}. Computation of $g^i(t)$ requires two matrix-vector multiplications, namely $A^i \, x(t)$ and $(A^i)^T \, \left(A^i \, x(t) - b^i\right)$, in that order. As $A^i$ is an $(n_i \times d)$-dimensional matrix and $x(t)$ is a $d$-dimensional vector, computation of gradient $g^i(t)$ requires $\bigO(n_i d)$ flops. Recall, from~\eqref{eqn:Rij}, that for each $j \in \{1, \ldots, \, d\}$,
\begin{align*}
        R^i_j(t) \hspace{-0.2em} & = \hspace{-0.2em}(A^i)^T \hspace{-0.2em}A^i k_j(t) + \dfrac{\beta}{m} k_j(t) \hspace{-0.1em} -  \dfrac{1}{m} e_j.
    \end{align*}
Thus, computation of each vector $R^i_j(t)$ requires two matrix-vector multiplications, namely $A^i \, k_j(t)$ and $(A^i)^T(A^i k_j(t))$, in that order. As $A^i$ is an $(n_i \times d)$-dimensional matrix, and both vectors $k_j(t)$ and $A^i \, k_j(t)$ are of dimensions $d$, computation of each $R^i_j(t)$ requires $\bigO(n_i d)$ flops. Thus, net computation of $d$ vectors $\{\R^i_j(t) : ~ j=1,...,d\}$ requires $\bigO(n_i d^2)$ flops. Therefore, the computational complexity of Algorithm~\ref{algo_1} for each agent $i$ is $\bigO(n_i d^2)$ flops, for each iteration. Note that, the computation of each member in the set $\{\R^i_j(t) : ~ j=1,...,d\}$ is independent of each other. Hence, agent $i$ can compute the $d$ vectors $\{\R^i_j(t) : ~ j=1,...,d\}$ in parallel. \\

For each iteration $t$, the server computes the matrix $K(t+1)$, defined in~\eqref{eqn:kcol_update}, and vector $x(t+1)$, defined in~\eqref{eqn:x_update}. Note that the computation of $K(t+1)$ only requires $\bigO(d)$ floating-point additions, and thus, can be ignored. In~\eqref{eqn:x_update}, the computation of $K(t+1) \sum_{i=1}^m g^i(t)$ requires only one matrix-vector multiplication between the $d \times d$ dimensional matrix $K(t+1)$ and the $d$-dimensional vector $\sum_{i=1}^m g^i(t)$. Thus, computation of $x(t+1)$ requires $\bigO(d^2)$ flops. Therefore, the computational complexity of Algorithm~\ref{algo_1} for the server is $\bigO(d^2)$ flops, for each iteration. Next, we present the formal convergence guarantees for Algorithm~\ref{algo_1}.

\subsection{Convergence Guarantees}
\label{sub:conv}
For a formal presentation of the convergence for Algorithm~\ref{algo_1}, we make a few elementary observations and define some notations below. 

\begin{itemize}
\setlength\itemsep{0.5em}
   \item Define the {\em collective observation vector} as
   \begin{align}
    B = \begin{bmatrix} (B^1)^T, \ldots, \, (B^m)^T \end{bmatrix}^T. \label{eqn:b_vector}
   \end{align}
    \item As the matrix $A^T A$ is positive semi-definite, if $\beta > 0$ then the matrix $\left(\A+\beta I\right)$ is positive definite, and therefore, invertible. We define
    \begin{align}
        K_{\beta} = \left(\A+\beta I\right)^{-1}. \label{eqn:def_k_beta}
    \end{align}
    The eigenvalues of matrix $\A$ are non-negative. Let $\lambda_1, \ldots, \, \lambda_d$ denote the eigenvalues of $\A$ such that $\lambda_1 \geq \ldots \geq \lambda_d \geq 0$. 
    
    \item Let the rank of matrix $\A$ be $r$. The value of $r$ is equal to $d$ if and only if the matrix $A$ is full column rank. In general, when $\A$ is not the trivial zero matrix, $1 \leq r \leq d$. Note that if $r < d$ then
    \begin{align}
        \lambda_1 \geq \ldots \geq \lambda_r > \lambda_{r + 1} = \ldots = \lambda_d = 0. \label{eqn:rank_eig}
    \end{align}
    
     \item For a matrix $M \in \R^{d \times d}$, let $\norm{M}_F$ denote its Frobenius norm, which is defined as the square root of the sum of squares of its elements~\cite{meyer2000matrix}. Specifically, if $m_{ij}$ denotes the $(i, \, j)$-th element of matrix $M$ then
     \begin{align}
         \norm{M}_F = \sqrt{\sum_{i=1}^d \sum_{j=1}^d m_{ij}^2}. \label{eqn:frob_norm}
     \end{align}

\end{itemize}

For each agent $i$, recall from~\eqref{eqn:dist_opt}, the cost function $F^i(x)$ is convex. Thus, the aggregate cost function $\sum_{i = 1}^m F^i(x)$ is also convex. Therefore, a point $x^* \in X^*$ if and only if
\[\nabla \sum_{i = 1}^m F^i(x^*) = 0_d,\]
where $0_d$ denotes the $d$-dimensional zero vector. For each iteration $t$, let $g(t)$ denote the gradient of the aggregate cost function $\sum_{i = 1}^m F^i(x)$ at $x = x(t)$. Recall, from~\eqref{eqn:g_i}, that for each $i$, $g^i(t) = \nabla F^i(x(t))$. Then,
\begin{align}
    g(t) = \nabla \sum_{i = 1}^m F^i(x(t)) = \sum_{i = 1}^m g^i(t). \label{eqn:grad_1}
\end{align}
The parameters defined below determine the minimum {\em rate of convergence} of Algorithm~\ref{algo_1}. Let,
\begin{align}
    \mu^* & = \frac{\lambda_1 - \lambda_r }{\lambda_1 + \lambda_r + 2 (\lambda_1 \lambda_r/ \beta)}, \text{ and } \label{eqn:opt_mu} \\
    \varrho & = \frac{\lambda_1 - \lambda_d }{\lambda_1 + \lambda_d + 2\beta}. \label{eqn:opt_rho}
\end{align}
We now present below the key result in the form of Theorem~\ref{thm:thm1}, on the convergence of Algorithm~\ref{algo_1}. 

\begin{theorem} \label{thm:thm1}
Consider Algorithm~\ref{algo_1}. If 
\begin{align}
    0 < \alpha < \frac{2}{\lambda_1 + \beta} ~ \text{, and }  ~ 0 < \delta < 2\left(\frac{\lambda_1 + \beta}{\lambda_1}\right), \label{eqn:alpha_delta}
\end{align}
then there exists non-negative real values $\mu$ and $\rho$ with
\begin{align}
    \mu^* \leq \mu < 1 ~ \text{, and } ~ \varrho \leq \rho < 1, \label{eqn:mu_rho}
\end{align}
such that the following hold true.
\begin{itemize}
    \item[(i)] For each iteration $t \geq 0$, 
    \begin{align}
        \norm{g(t+1)} ~ \leq ~ \left(\mu + \delta \lambda_1 \norm{K(0)-K_{\beta}}_F \, \rho^{t+1}\right) \, \norm{g(t)}, \label{eqn:conv_1}
    \end{align}
    where $\mu = \mu^*$ if 
    \begin{align}
        \delta = \frac{2}{\frac{\lambda_1}{\lambda_1 + \beta} + \frac{\lambda_r}{\lambda_r + \beta}}. \label{eqn:critical_delta}
    \end{align}
    \item[(ii)] For every $\epsilon > 0$ there exists a positive integer $N_{\epsilon}$ such that 
    \begin{align}
        \norm{g(t)} \leq \epsilon, \, \forall t \geq N_\epsilon. \label{eqn:stop}
    \end{align}
\end{itemize}
\end{theorem}

The proof of Theorem~\ref{thm:thm1} is deferred to Section~\ref{sec:analysis}.\\




As $\rho < 1$ (see~\eqref{eqn:mu_rho}), part (i) of Theorem~\ref{thm:thm1} implies that 
\begin{align}
    \underset{t \rightarrow \infty}{\lim}\frac{\norm{g(t+1)}}{\norm{g(t)}} \leq \mu < 1. \label{eqn:conv_rate}
\end{align}
Thus, part (ii) of Theorem~\ref{thm:thm1}, in conjunction with~\eqref{eqn:conv_rate}, implies that the sequence of gradients $\{g(t)\}_{t \geq 0}$ converges linearly to $0_d$ with {\em rate of convergence} equal to $\mu$. Since $g(t)$ is linearly related to $x(t)$ as presented in~\eqref{eqn:grad_1}, linear convergence of $\{g(t)\}_{t \geq 0}$ to $0_d$ implies linear convergence of the sequence of estimators $\{x(t)\}_{t \geq 0}$ to a minimum of the aggregate cost $\sum_{i = 1}^m F^i(x)$, in other words, to a point in $X^*$.\\

{\bf Superlinear convergence}: Next, we consider the special case when $x^*$ is the unique solution for the optimization problem defined in~\eqref{eqn:opt_1}. In other words, the aggregate cost function $\sum_{i = 1}^m F^i(x)$ has a unique minimum point. In this particular case, the matrix $\A$ is full-rank, and therefore, $r=d$.
Here, we will show that Algorithm~\ref{algo_1} with parameter $\beta = 0$ converges {\em superlinearly} to the minimum point $x^*$. Recall, from~\eqref{eqn:opt_rho}, that when $\beta = 0$ then 
\[\varrho = \frac{\lambda_1 - \lambda_d}{\lambda_1 + \lambda_d} < 1.\]
Specifically, we obtain the following corollary of Theorem~\ref{thm:thm1}. 

\begin{corollary} \label{cor:super_conv}
Consider Algorithm 1 with $\beta = 0$. If $x^*$ defined by~\eqref{eqn:opt_1} is unique, and the parameter $\alpha$ satisfies the condition stated in~\eqref{eqn:alpha_delta}, then for $\delta = 1$ there exists a non-negative real value $\rho \in [\varrho, \, 1)$ such that, for each iteration $t \geq 0$, 
\begin{align}
    \norm{g(t+1)} ~ \leq ~ \lambda_1 \norm{K(0)-K_{\beta}}_F \, \rho^{t+1} \,  \norm{g(t)}. \label{eqn:zperp_7}
\end{align}
\end{corollary}

The proof of Corollary~\ref{cor:super_conv} is deferred to Appendix~\ref{prf:super_conv}.\\


Since $\rho < 1$, Corollary~\ref{cor:super_conv} implies that the sequence of aggregate gradients $\{g(t)\}_{t \geq 0}$ converge to $0_d$ with rate of convergence equal to 
\[\lim_{t \to \infty} \frac{\norm{g(t+1)}}{\norm{g(t)}} \leq \lim_{t \to \infty} \lambda_1 \norm{K(0)-K_{\beta}}_F \, \rho^{t+1} =  0.\]
In other words, Algorithm~\ref{algo_1} converges {\em superlinearly} to the solution $x^*$ defined in~\eqref{eqn:opt_1}.\\

In the subsequent section, we discuss comparisons between the convergence of Algorithm~\ref{algo_1} and other existing methods, when solving the considered least-squares problem in distributed server-agent settings.

\newpage
\section{Comparisons with the Existing Methods}
\label{sec:comp}

In this section, we present comparisons between the optimum (smallest) rate of convergence of Algorithm~\ref{algo_1} with the server-agent based distributed versions of the following related algorithms:
\begin{itemize}
    \item Gradient-Descent~\cite{bertsekas1989parallel},
    \item Nesterov's Accelerated Gradient-Descent~\cite{nesterov27method},
    \item Heavy-Ball Method~\cite{polyak1964some},
    \item Accelerated Projection-Consensus (APC)~\cite{azizan2019distributed},
    \item Broyden–Fletcher–Goldfarb–Shanno (BFGS)~\cite{kelley1999iterative}.
\end{itemize}
The presented theoretical comparisons are verified through experiments on real data-sets in Section~\ref{sec:exp}.

\subsection{Gradient-Descent}

Consider the gradient-descent algorithm in server-agent networks, described in Section~\ref{sub:dgd}. As we have pointed out this algorithm to be equivalent to its centralized version, both of them have identical rate of convergence. In literature, the rate of convergence for centralized gradient-descent is known only when the solution for~\eqref{eqn:opt_1} is unique~\cite{lessard2016analysis, fazlyab2018analysis}. We present below, formally in Lemma~\ref{lem:gd}, the convergence of the gradient-descent algorithm in a server-agent network for the general case. We define a parameter 
\begin{align}
    \mu_{GD} = \dfrac{\lambda_1-\lambda_r}{\lambda_1+\lambda_r}. \label{eqn:mu_def}
\end{align}

\begin{lemma} \label{lem:gd}
Consider the gradient-descent algorithm in a server-agent network as presented in Section~\ref{sub:dgd}. In~\eqref{eqn:dgd}, if $\delta \in \left(0, \frac{2}{\lambda_1}\right)$ then there exists $\mu$ with $\mu_{GD} \leq \mu < 1$ such that, for each iteration $t \geq 0$, 
\begin{align}
    \norm{g(t+1)} ~ \leq ~ \mu \,  \norm{g(t)}. \label{eqn:gd_rate}
\end{align}
\end{lemma}
The proof of Lemma~\ref{lem:gd} is deferred to Appendix~\ref{prf:gd}.\\

We show formally below, in Theorem~\ref{thm:comp}, that Algorithm~\ref{algo_1} converges faster than the gradient-descent method in a server-agent network.
Recall that the largest and the smallest non-zero eigenvalues of the matrix $\A$ are denoted by $\lambda_1$ and $\lambda_r$.
Note that, in the special case when all the non-zero eigenvalues of the matrix $\A$ are equal, both the gradient-descent algorithm and Algorithm~\ref{algo_1} solve the optimization problem~\eqref{eqn:opt_1} in just one iteration. Now, Theorem~\ref{thm:comp} below presents the case when $\lambda_1 > \lambda_r$.

\begin{theorem} \label{thm:comp}
Consider Algorithm~\ref{algo_1}. Suppose that $\lambda_1>\lambda_r$. If $\beta > 0$ then there exists a positive finite integer $\tau$, and two positive finite real values $c$ and $r$ with $r < 1$, such that
\begin{align}
    \norm{g(t+1)} \leq c \left(r \, \mu_{GD} \right)^{t+1} \, \norm{g(0)}, \, \forall t > \tau. \label{eqn:z2_lem1}
\end{align}
\end{theorem}
The proof of Theorem~\ref{thm:comp} is deferred to Appendix~\ref{prf:comp}.\\

Now, consider the best possible rate of convergence for the gradient-descent algorithm in a server-agent network. That is, substitute $\mu = \mu_{GD}$ in Lemma~\ref{lem:gd}. In that case, we obtain the following upper bound on the gradients' norms for the gradient-descent algorithm in a server-agent network:
\begin{align}
    \norm{g(t+1)} \leq \left(\mu_{GD} \right)^{t+1} \, \norm{g(0)}, \, \forall t > 0. \label{eqn:gd_recur}
\end{align}
From Theorem~\ref{thm:comp}, we have an upper bound on the gradients' norm for Algorithm~\ref{algo_1} given by
\[\norm{g(t+1)} \leq c \left(r \, \mu_{GD} \right)^{t+1} \, \norm{g(0)}, \, \forall t > \tau.\]
Assuming that both the algorithms are identically initialized with some $x(0)$, we compare the ratio between the upper bounds on gradients for these algorithms, using~\eqref{eqn:z2_lem1} and~\eqref{eqn:gd_recur}.
We can see that there exists a finite integer $T$ such that
\[c \left(r \, \mu_{GD} \right)^{t+1} < \left(\mu_{GD} \right)^{t+1} \quad \forall t > T,\]
because $r < 1$. 
Alternately speaking, the ratio between the upper bounds on the gradients of Algorithm~\ref{algo_1} and gradient-descent in server-agent network is given by $c~r^{t+1}$ for iteration $t> T$, where $r<1$.
This statement implies that even though Algorithm~\ref{algo_1} might be initially slower than the gradient-descent algorithm in a server-agent network, after a finite number of iterations  Algorithm~\ref{algo_1} is guaranteed to have a smaller error bound compared to the gradient-descent in a server-agent network with identical initialization of $x(0)$ and arbitrary initialization of the iterative pre-conditioning matrix $K(0)$. More importantly, this error bound of Algorithm~\ref{algo_1} decreases to zero at an {\em exponentially faster} rate compared to the latter one.

\subsection{Nesterov's Accelerated Gradient-Descent}
\label{sub:nag}
In this subsection, we describe Nesterov's accelerated gradient-descent method in a server-agent network.
In this method, in addition to the estimate $x(t)$ of a minimum point, the server also maintains a {\em memory vector} denoted by the vector $y(t)$ for each iteration $t \in \{0, \, 1, \ldots\}$. The initial estimate $x(0)$ and initial memory vector $y(0)$ are chosen arbitrarily from $\R^d$. Also, before initiating the iterations, the server chooses two non-negative scalar parameters $\delta$ and $\eta$. In each iteration $t$, upon receiving the estimate $x(t)$ from the server, each agent computes the local gradient $g^i(t)$ defined by~\eqref{eqn:g_i} and sends it back to the server. The server, upon receiving the local gradients from all the agents, updates the memory vector $y(t)$ and the current estimate $x(t)$ as follows:
\begin{align}
    y(t+1) & = x(t) - \delta \sum_{i=1}^m g^i(t), \label{eqn:nag_1} \\
    x(t+1) & = (1+\eta) y(t+1) - \eta y(t). \label{eqn:nag_2}
\end{align}
As the actual gradient of the aggregate cost function $\sum_{i = 1}^n F^i(x)$ for each iteration $t$ is equal to the sum of all agents' gradients (see~\eqref{eqn:sum_grads}), the update pair~\eqref{eqn:nag_1}-\eqref{eqn:nag_2} above is equivalent to the centralized Nesterov's accelerated gradient-descent method~\cite{nesterov27method}. Thus, the convergence of the above implementation of the Nesterov's accelerated gradient-descent method in the server-agent network is equivalent to its centralized version presented in~\cite{nesterov27method}. The rate of convergence of the centralized Nesterov's accelerated gradient-descent method is known explicitly only for the special case when the optimization problem~\eqref{eqn:opt_1} has a unique solution~\cite{lessard2016analysis,fazlyab2018analysis}. For the particular case when~\eqref{eqn:opt_1} has a unique solution the Nesterov's accelerated gradient-descent method converges {\em linearly} with provably smaller rate of convergence than that of the traditional gradient-descent method. On the other hand, we have shown, in Corollary~\ref{cor:super_conv}, that Algorithm~\ref{algo_1} converges {\em superlinearly} when~\eqref{eqn:opt_1} has a unique solution.

\subsection{Heavy-Ball Method}
\label{sub:hbm}

Here, we describe the heavy-ball method in a server-agent network. In this method, instead of the memory vector as in Nesterov's accelerated gradient-descent, the server maintains a {\em momentum vector} which is denoted by $w(t)$ for each iteration $t \in \{0, \, 1, \ldots\}$. The initial estimate $x(0)$ and initial momentum vector $w(0)$ are chosen arbitrarily from $\R^d$. Also, before initiating the iterations, the server chooses two non-negative scalar parameters $\delta$ and $\eta$.
The local gradient $g^i(t)$ is computed by each agent $i$ as before. This algorithm is somewhat similar to the Nesterov's accelerated gradient-descent method in a server-agent network described earlier, except that the update equations at the server are different. Here, the server updates the momentum vector $w(t)$ and the current estimate $x(t)$ according to:
\begin{align}
    w(t+1) & = \eta w(t) + \sum_{i=1}^m g^i(t), \label{eqn:hbm_1} \\
    x(t+1) & = x(t) - \delta w(t+1). \label{eqn:hbm_2}
\end{align}
Following a similar argument as in Section~\ref{sub:nag}, we can conclude that the convergence of the above implementation~\eqref{eqn:hbm_1}-\eqref{eqn:hbm_2} of the heavy-ball method in the server-agent network is equivalent to the centralized heavy-ball method described in~\cite{polyak1964some}. Again, the explicit rate of convergence of the centralized heavy-ball method is known only for the special case of~\eqref{eqn:opt_1} having a unique solution. In that special case, the heavy-ball method has a linear rate of convergence~\cite{lessard2016analysis} that is provably smaller than the Nesterov's accelerated gradient-descent method. Whereas, we have shown in Corollary~\ref{cor:super_conv} that Algorithm~\ref{algo_1} converges {\em superlinearly} when~\eqref{eqn:opt_1} has a unique solution.

\subsection{Accelerated Projection-Consensus}
\label{sub:apc}

The accelerated projection-based consensus (APC) algorithm is applicable to a special case of the least-squares problem~\eqref{eqn:opt_1} when the collective algebraic equations $Ax=B$ has a unique solution. In addition, all the local data matrices $A^i$ needs to be full row-rank. Here, each agent $i$ maintains a local estimate of the minimum point, denoted by $x^i(t)$, and the server maintains a global estimate denoted by $x(t)$. Before initiating the iterations, the server chooses two non-negative scalar parameters $\eta$ and $\gamma \leq 2$ and communicates the parameter $\gamma$ to all the agents. Additionally, each agent computes its projection matrix onto the nullspace of the matrix $A^i$ as
\[P^i = I - (A^i)^T \left(A^i (A^i)^T\right)^{-1}A^i.\]
Based on its data pair $(A^i,B^i)$ , each agent $i$ initializes its local estimate $x^i(0)$ as one of the solutions of $A^ix=B^i$ and sends it to the server. The server computes the average initialized local estimate of all the agents and sets the global initial estimate of the minimum point as
\[x(0) = \frac{1}{m} \sum_{i=1}^m x^i(0).\]
Then, at each iteration $t \in \{0, \, 1, \ldots\}$, each agent receives the current global estimate $x(t)$ from the server and updates its local estimate $x^i(t)$ according to
\[x^i(t+1) = x^i(t) + \gamma P^i \left (x(t)-x^i(t)\right).\]
The server receives the updated local estimates from all the agents and updates the current global estimate $x(t)$ as follows:
\[x(t+1) = \frac{\eta}{m} \sum_{i=1}^m x^i(t) + (1-\eta)x(t).\]
This algorithm has been shown to converge to the minimum point {\em linearly} and speculated to be faster than the heavy-ball method described above. However, we have shown that Algorithm~\ref{algo_1} converges {\em superlinearly} in this particular case of~\eqref{eqn:opt_1} where the APC method is applicable.

\subsection{Broyden–Fletcher–Goldfarb–Shanno (BFGS)}
\label{sub:bfgs}

In this subsection we present the algorithm due to Broyden–Fletcher–Goldfarb–Shanno, popularly known as BFGS method, when applied in a sever-agent network. Note that, this method is only applicable for the special case when the least-squares problem~\eqref{eqn:opt_1} has a unique solution. The BFGS method is a {\em quasi-Newton} iterative method in which the server maintains a square matrix $M(t)$ of dimension $d\times d$ that approximates the Hessian matrix of the aggregate cost function $\sum_{i = 1}^mF^i(x)$ which is equal to $(A^T A)^{-1}$. The initial estimate $x(0)$ is chosen arbitrarily, whereas the matrix $M(0)$ is initialized as any non-singular matrix of appropriate dimension. In addition, the server selects a stepsize parameter $\eta(t)$ for each iteration $t$ using a line search method~\cite{kelley1999iterative}. As usual, each agent receives the current estimate $x(t)$ from the server and computes the gradient $g^i(t)$ as per~\eqref{eqn:g_i}. The server accumulates all the agents' gradients and performs the following two steps at each iteration $t$.
\begin{itemize}
    \item Obtain a vector $s(t)$ by solving the equations 
    $$M(t)\,s(t) = - \sum_{i=1}^m g^i(t).$$
    \item Update the estimate as 
    $$x(t+1) = x(t) + \eta(t)s(t).$$
\end{itemize}
During the same iteration, the server broadcasts the updated estimate $x(t+1)$ to all the agents. Each agent then computes the gradient $g^i(t+1)$ and sends it back to the server. Next, the server updates the approximate Hessian matrix as follows:
\begin{align*}
    y(t) & = \sum_{i=1}^m g^i(t+1) - \sum_{i=1}^m g^i(t), \\
    M(t+1) & = M(t) + \frac{y(t)y(t)^T}{\eta(t) y(t)^T s(t)} - \frac{M(t)s(t)s(t)^T M(t)^T}{s(t)^T M(t) s(t)}. 
\end{align*}
Following a similar argument as in Section~\ref{sub:nag}, we can conclude that the BFGS method in a server-agent network is equivalent to its centralized version in~\cite{kelley1999iterative}. Again, the rate of convergence of the centralized BFGS is known to be superlinear~\cite{kelley1999iterative}. By equivalence of these algorithms, the BFGS method in a server-agent network has superlinear rate of convergence only for unique solution of Algorithm~\ref{algo_1}.

\subsection{Conclusion}

We first consider the special case when the solution for the distributed least-squares problem~\eqref{eqn:opt_1} is unique, i.e, the global cost function $\sum_{i = 1}^m F_i(x)$ is strongly convex. As shown above in Section~\ref{sub:conv}, Algorithm~\ref{algo_1} converges {\em superlinearly} in this case to the minimum point. However, the server-agent versions of the algorithms gradient-descent, Nesterov's accelerated gradient-descent, heavy-ball method, and the APC method can only converge {\em linearly} to the minimum point. \\

Next, we consider the more general case when the solution for the distributed least-squares problem~\eqref{eqn:opt_1} is not unique. We note that the rates of convergence of the server-agent versions of the gradient-descent, Nesterov's accelerated gradient-descent, heavy-ball method, BFGS, and the APC method, are known only for the case when~\eqref{eqn:opt_1} has a unique solution, as have been discussed in this section. We formally presented the convergence of the gradient-descent method in a server-agent network for the general case, which has helped us in showing that Algorithm~\ref{algo_1} converges exponentially faster than gradient-descent in the general case where the aggregate cost is non-strongly convex.



\newpage
\section{Extension to Convex Quadratic Costs}
\label{sec:quad}

We have considered the multi-agent distributed linear least-squares problem~\eqref{eqn:opt_1} where each agent's local cost function $F^i(x)$ is in squared form and they collaborate with a server to minimize the aggregate cost function $\sum_{i=1}^m F^i(x)$. This section considers a more general case of each agent's cost function being quadratic and convex.
\\

We consider, again, a server-agent based distributed system of $m$ agents and a server where each agent holds a $(d \times d)$-dimensional real valued matrix $P^i$, a $d$-dimensional real valued column vector $q^i$ and a scalar real number $r^i$. Then, for each agent $i$, $P^i \in \R^{d \times d}$, $q^i \in \R^d$ and $r^i\in \R$. The agents aim to compute a parameter vector $x^* \in \R^d$ such that
\begin{align}
    x^* \in X^* = \arg \min_{x \in \R^d} \sum_{i = 1}^m \left(\frac{1}{2}x^T P^i x - x^T q^i + r^i \right). \label{eqn:opt_2}
\end{align}
Here, each agent's local cost is given by
\begin{align}
    F^i(x) = \frac{1}{2}x^T P^i x - x^T q^i + r^i, \quad \forall x \in \R^d,  \label{eqn:quad_local}
\end{align}
which is quadratic in its argument. By defining the notations
\[P = \sum_{i = 1}^m P^i, \quad q = \sum_{i = 1}^m q^i, \quad r = \sum_{i = 1}^m r^i,\]
the aggregate cost function in~\eqref{eqn:opt_2} can be rewritten as $\sum_{i=1}^m F^i(x)$ where 
\begin{align}
    \sum_{i=1}^m F^i(x) = \frac{1}{2}x^T P x - x^T q + r, \quad \forall x \in \R^d.  \label{eqn:quad_agg}
\end{align}
In this section, we only consider the case where the aggregate cost in the quadratic minimization problem~\eqref{eqn:opt_2} is convex. Since the Hessian of $\sum_{i=1}^m F^i(x)$ is the constant matrix $P$, the aggregate cost in~\eqref{eqn:opt_2} is convex if and only if the matrix $P$ is positive semi-definite. For the remainder of this section, we will assume that the matrix $P$ is positive semi-definite.
\\

Here we argue that the convex quadratic problem~\eqref{eqn:opt_2} can be seen as a special case of the least-squares problem~\eqref{eqn:opt_1}. In order to see this, we expand the local cost in~\eqref{eqn:opt_1} as follows:
\begin{align}
    \frac{1}{2}\norm{A^i x - B^i}^2 = \frac{1}{2}x^T (A^i)^T A^i x - x^T (A^i)^T B^i + \frac{1}{2} (B^i)^T B^i, \quad \forall x \in \R^d.  \label{eqn:loc_expand}
\end{align}
Comparing~\eqref{eqn:quad_local} and~\eqref{eqn:loc_expand}, we can recover the cost function in~\eqref{eqn:opt_2} from that of~\eqref{eqn:opt_1} by setting
\begin{align}
    (A^i)^T A^i = P^i, \quad (A^i)^T B^i = q^i, \quad \frac{1}{2}(B^i)^T B^i = r^i, \quad i=1,\ldots,m. \label{eqn:par_set}
\end{align}
Hence, if the matrix $P$ is positive semi-definite, then the convex quadratic problem~\eqref{eqn:opt_2} is a special case of~\eqref{eqn:opt_1}, with the parameter replacements mentioned in~\eqref{eqn:par_set}.
A particular case of interest is the matrix $P=\sum_{i = 1}^m P^i$ being positive definite in~\eqref{eqn:opt_2}. In this case, the aggregate cost $\sum_{i=1}^m F^i(x)$ in~\eqref{eqn:quad_agg} is strictly convex, and hence the optimization problem~\eqref{eqn:opt_2} has a unique solution.
\\

We propose to solve the convex quadratic problem~\eqref{eqn:opt_2} in server-agent network using Algorithm~\ref{algo_1} with the parameter replacements in~\eqref{eqn:par_set}.
Owing to the above discussion, all of the convergence guarantees of Algorithm~\ref{algo_1} (see Section~\ref{sub:conv}) and its theoretical comparison with related algorithms (see Section~\ref{sec:comp}) for the linear least-squares problem~\eqref{eqn:opt_1}, specifically Theorem~\ref{thm:thm1}, Corollary~\ref{cor:super_conv} and Theorem~\ref{thm:comp}, are exactly applicable to the general convex quadratic problem~\eqref{eqn:opt_2}.

\newpage

\section{Proof of Theorem~\ref{thm:thm1}}
\label{sec:analysis}

In this section, we present the proof for Theorem~\ref{thm:thm1}. Throughout this section we assume that $\A$ is not the trivial zero matrix, $\beta > 0$, and 
\begin{align}
    0 < \alpha < \frac{2}{\lambda_1+\beta}. \label{eqn:alpha_bnd}
\end{align}

The proof relies on the following lemma, Lemma~\ref{lem:k}, which shows the {\em linear} convergence of the sequence of matrices $\{K(t), ~ t= 0, \, 1, \ldots \}$ to $K_{\beta}$. Let $k_{j\beta}$ denote the $j$-th column of matrix $K_{\beta}$ where $j=1, \ldots, \, d$, and recall the definition of $\varrho$ from~\eqref{eqn:opt_rho}.

\begin{lemma} \label{lem:k}
Consider Algorithm~\ref{algo_1}. If $\alpha \in \left(0,\frac{2}{\lambda_1+\beta}\right)$ then there exists $\rho$ with $\varrho \leq \rho < 1$ such that for each $j \in \{1, \ldots, \, d\}$,
\begin{align*}
    \norm{k_j(t+1)-k_{j\beta}} \leq \rho \norm{k_j(t)-k_{j\beta}}, \quad \forall t \in \{0, \, 1, \ldots \}.
\end{align*}
\end{lemma}
The proof of Lemma~\ref{lem:k} is deferred to Appendix~\ref{prf:k}.\\






Now, for each $i$ and $t$, upon substituting $g^i(t)$ from~\eqref{eqn:g_i} in~\eqref{eqn:grad_1} we obtain that
\begin{align}
    g(t) = \sum_{i = 1}^m g^i(t) =  \left( \sum_{i=1}^m (A^i)^T A^i \right) x(t) - \left( \sum_{i=1}^m (A^i)^T B^i \right). \label{eqn:grad_2}
\end{align}
As $A^T A = \sum_{i=1}^m (A^i)^T A^i$ and $A^T B = \sum_{i=1}^m (A^i)^T B^i$,~\eqref{eqn:grad_2} implies that
\begin{align}
    g(t) = \sum_{i = 1}^m g^i(t) = A^T (Ax(t)-B). \label{eqn:grad}
\end{align}
Recall the definition of $X^*$ from~\eqref{eqn:opt_1}. Due to~\eqref{eqn:grad},
\begin{align}
    X^* = \left\{ x \in \R^d ~ : ~ A^T (A x - B) = 0_d\right\}. \label{eqn:new_X*}
\end{align}
Consider an arbitrary point $x^* \in X^*$. Define 
\begin{align}
    z(t) = x(t) - x^*. \label{eqn:zt}
\end{align}
As $A^T (A x^* - B) = 0_d$ (due to~\eqref{eqn:new_X*}), upon substituting from~\eqref{eqn:zt} in~\eqref{eqn:grad} we obtain that
\begin{align}
    g(t) = \A \, z(t). \label{eqn:grad_zt}
\end{align}
Let $\N(\A)$ denote the nullspace of matrix $\A$:
\[\N(\A) = \left\{x \in \R^d ~: ~ \A x = 0_d \right\}.\]
Let $\mathcal{N}(\A)^{\perp}$ denote the orthogonal vector space of $\N(\A)$:
\[\mathcal{N}(\A)^{\perp} = \left\{x \in \R^n ~ : ~ x^T y = 0, ~ \forall y \in \N(\A) \right\}.\]
Due to the fundamental theorem of linear algebra~\cite{horn2012matrix}, $\R^n = \mathcal{N}(\A) \oplus \mathcal{N}(\A)^{\perp}$. Therefore, for each $t \geq 0$, we can decompose vector $z(t)$ into two orthogonal vectors $z(t)^{\perp}$ and $z(t)^{\N}$, such that $z(t)^{\perp} \in \mathcal{N}(\A)^{\perp}$ and $z(t)^{\N} \in \N(\A)$. Specifically, for each $t \geq 0$,
\begin{align}
    z(t) = z(t)^{\mathcal{N}} + z(t)^{\perp}. \label{eqn:decomp}
\end{align}
As $\A \, z(t)^{\N} = 0_d$, upon substituting from~\eqref{eqn:decomp} in~\eqref{eqn:grad_zt} we obtain that
\begin{align}
    g(t) = \A \, z(t)^{\perp}, \quad \forall t \geq 0. \label{eqn:grad_vs_err}
\end{align}
The remainder of the proof is divided into three steps. In the first two steps, we will prove part (i) of the theorem. In the last step, we will prove part (ii). \\


{\bf Step I:} For each iteration $t$, let $\widetilde{K}(t)$ denote the matrix obtained by stacking the column vectors $\widetilde{k}_1(t), \ldots, \,  \widetilde{k}_d(t)$. Specifically, 
\begin{align}
    \widetilde{K}(t) = \left[\widetilde{k}_1(t), \ldots, \, \widetilde{k}_d(t) \right] = K(t) - K_{\beta}, \quad \forall t \geq 0. \label{eqn:tilde_k}
\end{align}
In this step, we will show that 
\begin{align}
    g(t+1) = \left(I - \delta \A Q \, K_{\beta} \right)\, g(t) - \delta \A Q \widetilde{K}(t+1) \, g(t), \label{eqn:gt_dyn_1}
\end{align}
for an appropriate projection matrix $Q$ defined later in this step.\\

Upon substituting from~\eqref{eqn:grad} in~\eqref{eqn:x_update} we obtain that
\begin{align}
    x(t+1) = x(t) - \delta K(t+1) A^T\left(Ax(t)- B \right), \quad \forall t \geq 0. \label{eqn:x_central}
\end{align}
Upon substituting from~\eqref{eqn:zt} in~\eqref{eqn:x_central}, due to~\eqref{eqn:new_X*} we obtain that
\begin{align}
    z(t+1) = \left(I- \delta K(t+1)\A\right) \, z(t). \label{eqn:z_mltp}
\end{align}
Being a symmetric positive semi-definite matrix, $\A$ has real non-negative eigenvalues and a corresponding set of $d$ orthonormal eigenvectors. Recall that $\lambda_1, \ldots, \, \lambda_d$ denote the eigenvalues of $\A$ such that $\lambda_1 \geq \, \ldots \, \geq \lambda_d \geq 0$. We denote by $Diag(.)$ a diagonal matrix of appropriate dimensions, with the arguments denoting the diagonal elements of the matrix in the same order.
Let $S = Diag(\lambda_1, \ldots \,,\lambda_d)$, and let the matrix $V$ consists of the corresponding orthonormal eigenvectors $[V_1, \ldots, \, V_d]$ such that 
\[\A V_j = \lambda_j \, V_j.\]
Note that $V_i^T V_i = 1$ and $V^T_i V_j = 0$ for all $i \neq j$. Then $V^T V = I$, and
\begin{align}
    \A = V S V^T. \label{eqn:svd}
\end{align}
Recall that $r$ denotes the rank of matrix $\A$. In general, $1 \leq r  \leq d$. If $r < d$ then 
\[\lambda_1 \geq \, \ldots \, \geq \lambda_r > \lambda_{r+1} = \, \ldots \, = \lambda_d = 0.\]
If $r = d$ then all the eigenvalues of $\A$ are positive. Thus, 
\begin{align}
    S = Diag\left( \lambda_1, \ldots, \lambda_r, \, \underbrace{0, \ldots, 0}_{d-r}\right). \label{eqn:diag_S}
\end{align}

Let $span\left\{V_1, \ldots, \, V_r \right\}$ denote the vector space spanned by the orthonormal eigenvectors  $V_1, \ldots, \, V_r$: 
\[span\left\{V_1, \ldots, \, V_r \right\} = \left\{ \sum_{i = 1}^r u_i V_i ~ : ~ u_i \in \R, ~ \forall i \right\}.\]
As the eigenvectors $V_1, \ldots, \, V_d$ are orthogonal~\cite{horn2012matrix}, 
\begin{align}
\begin{split}
    \N(\A)^{\perp} & = span\left\{ V_1, \ldots, V_r \right\}, \text{ and } \\ \label{eqn:orth_space_r}
    \N(\A) & = span\left\{ V_{r+1}, \ldots, V_d \right\}.
\end{split}
\end{align}
Let,
\[ S^{\perp} = Diag\left( \underbrace{1,\dots, 1}_r,  \, \underbrace{0, \ldots, 0}_{d-r} \right). \]
Define a {\em projection matrix} 
\begin{align}
    Q = V S^{\perp} V^T.  \label{eqn:def_Q}
\end{align}
Note that for a vector $v \in \R^d$, due to the fundamental lemma of linear algebra, the vectors $Q v$ and $(v - Qv)$ belong to the orthogonal vector spaces $\N(\A)^{\perp}$ and $\N(\A)$, respectively (cf.~\eqref{eqn:orth_space_r}). Thus, from the definition of $z(t)^{\perp}$ in~\eqref{eqn:decomp}, for all $t$,
\begin{align}
    z(t)^{\perp} = Q z(t). \label{eqn:z_per_Q}
\end{align}
This implies that, for all $t \geq 0$,
\begin{align}
    z(t+1)^{\perp}  = Q \, z(t+1) & \overset{\eqref{eqn:z_mltp}}{=} Q \, \left(I - \delta K(t+1)\A  \right) \,z(t) \overset{\eqref{eqn:z_per_Q}}{=}  z(t)^{\perp} - \delta Q K(t+1)\A \, z(t) \nonumber \\
    & \overset{\eqref{eqn:decomp}}{=} z(t)^{\perp} - \delta Q K(t+1)\A \, \left(z(t)^{\mathcal{N}} + z(t)^{\perp} \right). \label{eqn:zt_dyn} 
\end{align}
As $z(t)^{\mathcal{N}} \in {\mathcal{N}}(\A)$, $\A \, z(t)^{\mathcal{N}} = 0_d$. Upon substituting this in~\eqref{eqn:zt_dyn} we obtain that 
\begin{align}
    z(t+1)^{\perp} = z(t)^{\perp} - \delta Q K(t+1)\A \, z(t)^{\perp}, \quad \forall t. \label{eqn:zt_dyn_2}
\end{align}
Substituting from~\eqref{eqn:grad_vs_err} in~\eqref{eqn:zt_dyn_2} we obtain that
\begin{align}
    z(t+1)^{\perp} = z(t)^{\perp} - \delta Q K(t+1) \, g(t). \label{eqn:zt_dyn_3}
\end{align}
Multiplying both sides of~\eqref{eqn:zt_dyn_3} with $\A$, and substituting again from~\eqref{eqn:grad_vs_err}, we obtain that
\begin{align}
    g(t+1) = g(t) - \delta \, \A Q \, K(t+1) \, g(t). \label{eqn:zt_dyn_4}
\end{align}
Finally, substituting from~\eqref{eqn:tilde_k} in~\eqref{eqn:zt_dyn_4} proves~\eqref{eqn:gt_dyn_1}, which is
\[g(t+1) = \left(I - \delta \A Q \, K_{\beta} \right)\, g(t) - \delta \A Q \widetilde{K}(t+1) \, g(t), ~ ~ \forall t \geq 0.\]
~\\

{\bf Step II:} Using triangle inequality in~\eqref{eqn:gt_dyn_1} we obtain that 
\begin{align}
    \norm{g(t+1)} \leq \norm{\left(I -  \delta \A Q \, K_{\beta} \right) \, g(t)} + \delta \norm{\A Q \widetilde{K}(t+1) g(t)}. \label{eqn:zperp_1}
\end{align}
In this step, we will show part (i) of the theorem, which states that for all $t \geq 0$,
\begin{align*}
    \norm{g(t+1)} \leq \left( \mu + \delta \lambda_1 \norm{\widetilde{K}(0)}_F \, \rho^{t+1} \right) \, \norm{g(t)}.
\end{align*}

First, we will derive an upper bound on the second term in~\eqref{eqn:zperp_1}.
Recall the definition of $\widetilde{K}(t)$ from~\eqref{eqn:tilde_k}. Due to Lemma~\ref{lem:k}, for each $j \in \{1, \ldots, \, d\}$ we obtain that
\begin{align}
    \norm{\widetilde{k}_j(t)}^2 \leq \rho^{2t} \norm{\widetilde{k}_j(0)}^2. \label{eqn:norm_tilde_K}
\end{align}
Note that 
\begin{align}
    \norm{\widetilde{K}(t)}_F^2 = \sum_{j=1}^d \norm{\widetilde{k}_j(t)}^2. \label{eqn:norm_frob_K}
\end{align}
Upon substituting from~\eqref{eqn:norm_tilde_K} above, and applying square-roots on both sides, we obtain that for all $t \geq 0$,
\begin{align}
    \norm{\widetilde{K}(t)}_F \leq \rho^{t} \norm{\widetilde{K}(0)}_F. \label{eqn:K_frob_0}
\end{align}
For any square matrix $M$, let $\norm{M}$ denote the {\em induced 2-norm} of the matrix. Then~\cite{horn2012matrix}, 
\begin{align}
    \norm{M} \leq \norm{M}_F. \label{eqn:matrix_norms}
\end{align}
Substituting from~\eqref{eqn:K_frob_0} in~\eqref{eqn:matrix_norms} we obtain that, for all $t \geq 0$,
\begin{align}
    \norm{\widetilde{K}(t)} \leq \rho^{t} \norm{\widetilde{K}(0)}_F. \label{eqn:K_frob}
\end{align}
Recall that $V$ is a unitary matrix. So, $V^T V = I$. Thus, due to~\eqref{eqn:svd} and~\eqref{eqn:def_Q}, 
\begin{align*}
    \A Q = V \, \left( S S^{\perp} \right) \, V^T.
\end{align*}
Recall that $\lambda_1$ denotes the largest eigenvalue of $\A$. Thus, the largest element in the diagonal matrix $S S^{\perp}$ has value equal to $\lambda_1$. Therefore~\cite{horn2012matrix},
\begin{align}
    \norm{\A Q} = \lambda_1. \label{eqn:aQ_lambda}
\end{align}
From the definition of induced 2-norm~\cite{horn2012matrix}, for any vector $v \in \R^d$ and a $(d \times d)$ real-valued matrix $M$,
\begin{align*}
    \norm{M v} \leq \norm{M} \, \norm{v}. 
\end{align*}
Therefore,  
\begin{align*}
    \norm{\delta \A \, Q \widetilde{K}(t+1)g(t)} \leq \delta \norm{\A Q} \norm{\widetilde{K}(t+1)} \norm{g(t)} .
\end{align*}
Substituting from~\eqref{eqn:K_frob} and~\eqref{eqn:aQ_lambda} above we obtain that, for all $t$,
\begin{align}
    \norm{\delta \A \, Q \widetilde{K}(t+1)g(t)} \leq \delta \lambda_1 \norm{\widetilde{K}(0)}_F \, \rho^{t+1} \norm{g(t)}. \label{eqn:z_perp_rhs}
\end{align}
Next, we will derive an upper bound on the first term in~\eqref{eqn:zperp_1}. \\

Recall, from~\eqref{eqn:def_k_beta}, that
\[K_{\beta} = \left(\A + \beta I\right)^{-1}.\]
As $\A = V Diag\left(\lambda_1, \ldots, \, \lambda_d\right) V^T$ (see~\eqref{eqn:svd}) where the matrix $V$ comprising the orthonormal eigenvectors of $\A$ is unitary, satisfying $V V^T = I$, from above we obtain that
\begin{align}
    K_{\beta} = V \, Diag\left(\frac{1}{\lambda_1 + \beta} \, , \ldots, ~ \frac{1}{\lambda_d + \beta}\right) \, V^T. \label{eqn:decomp_k_b}
\end{align}
Also, from~\eqref{eqn:svd},~\eqref{eqn:diag_S}, and~\eqref{eqn:def_Q}, we have
\begin{align}
    \A Q = V \, Diag\left(\lambda_1, \ldots, \, \lambda_r, \, \underbrace{0, \ldots, 0}_{d-r}\right) \, V^T, \label{eqn:trunc_A}
\end{align}
where $r$ denotes the rank of matrix $\A$. From~\eqref{eqn:decomp_k_b} and~\eqref{eqn:trunc_A} we obtain that
\begin{align}
    \A Q K_{\beta} = V \, Diag\left(\frac{\lambda_1}{\lambda_1 + \beta}, \ldots, \, \frac{\lambda_r}{\lambda_r + \beta}, \, \underbrace{0, \ldots, 0}_{d-r}\right) \, V^T. \label{eqn:comb_matrices}
\end{align}
Substituting from~\eqref{eqn:comb_matrices} we obtain that
\begin{align}
    \left( I  - \delta \,  \A Q K_{\beta} \right) = V \, Diag \left(\left(1 -  \frac{\delta \lambda_1}{\lambda_1 + \beta}\right), \ldots, \, \left(1 - \frac{\delta \lambda_r}{\lambda_r + \beta} \right), \, \underbrace{1, \ldots, 1}_{d-r} \right) \, V^T . \label{eqn:all_matrices} 
\end{align}
This implies that
\begin{align}
    \norm{\left( I  - \delta \,  \A Q K_{\beta} \right) g(t)} = \norm{V \, Diag \left(\left(1 -  \frac{\delta \lambda_1}{\lambda_1 + \beta}\right), \ldots, \, \left(1 - \frac{\delta \lambda_r}{\lambda_r + \beta} \right), \, \underbrace{1, \ldots, 1}_{d-r} \right) \, V^T \, g(t)}. \label{eqn:all_matrices_g_t}
\end{align}
Now, note that from~\eqref{eqn:grad_vs_err}, $g(t) \in Im\left(\A\right)$ where $Im(\cdot)$ denotes the image of a matrix operator. Owing to the fundamental theorem of linear algebra~\cite{horn2012matrix}, $Im\left(\A\right) =  \N \left(\A \right)^{\perp}$. Thus,
\begin{align}
    g(t) \in \N(\A)^{\perp} = span\left\{ V_1, \ldots, \, V_r\right\}. \label{eqn:g_t_span}
\end{align}
Recall that the vectors $V_1, \ldots, \, V_d$, constituting the matrix $V$, are orthonormal. Therefore, due to~\eqref{eqn:g_t_span}, 
\begin{align}
    &\norm{V \, Diag \left(\left(1 -  \frac{\delta \lambda_1}{\lambda_1 + \beta}\right), \ldots, \, \left(1 - \frac{\delta \lambda_r}{\lambda_r + \beta} \right), \, \underbrace{1, \ldots, 1}_{d-r} \right) \, V^T \, g(t)} \leq  \nonumber \\
    & \max \left\{ \mnorm{1 -  \frac{\delta \lambda_1}{\lambda_1 + \beta}}, \ldots, \, \mnorm{1 - \frac{\delta \lambda_r}{\lambda_r + \beta}} \right\} \, \norm{g(t)}, \label{eqn:span_norm_g_t_0}
\end{align}
where $\mnorm{\cdot}$ denotes the absolute value. As $\lambda_1 \geq \ldots \geq \lambda_r > 0$ and $\beta > 0$, if 
\[0 < \delta < 2 \left(\frac{\lambda_1 + \beta}{\lambda_1} \right) = 2\left(1 + \frac{\beta}{\lambda_1}\right),\]
then
\begin{align}
    &\max \left\{ \mnorm{1 -  \frac{\delta \lambda_1}{\lambda_1 + \beta}}, \ldots, \, \mnorm{1 - \frac{\delta \lambda_r}{\lambda_r + \beta}} \right\} = \max \left\{ \mnorm{1 -  \frac{\delta \lambda_1}{\lambda_1 + \beta}}, \mnorm{1 - \frac{\delta \lambda_r}{\lambda_r + \beta}} \right\} \label{eqn:basic_ineq_1} \\
    \text{and} ~ & \mnorm{1 -  \frac{\delta \lambda_i}{\lambda_i + \beta}} ~ < 1, \, i=1,\ldots,r. \label{eqn:basic_ineq_2}
\end{align}
Substituting from~\eqref{eqn:basic_ineq_1} in~\eqref{eqn:span_norm_g_t_0} we obtain that
\begin{align}
    & \norm{V \, Diag \left(\left(1 -  \frac{\delta \lambda_1}{\lambda_1 + \beta}\right), \ldots, \, \left(1 - \frac{\delta \lambda_r}{\lambda_r + \beta} \right), \, \underbrace{1, \ldots, 1}_{d-r} \right) \, V^T \, g(t)} \leq \nonumber \\
    & \max \left\{ \mnorm{1 -  \frac{\delta \lambda_1}{\lambda_1 + \beta}}, \mnorm{1 - \frac{\delta \lambda_r}{\lambda_r + \beta}} \right\} \norm{g(t)}. \label{eqn:span_norm_g_t}
\end{align}
Substituting from~\eqref{eqn:span_norm_g_t} in~\eqref{eqn:all_matrices_g_t} we obtain that, for all $t$,
\begin{align}
    \norm{\left( I  - \delta \,  \A Q K_{\beta} \right) \, g(t) } \leq \max \left\{ \mnorm{1 -  \frac{\delta \lambda_1}{\lambda_1 + \beta}}, \mnorm{1 - \frac{\delta \lambda_r}{\lambda_r + \beta}} \right\} \, \norm{g(t)}. \label{eqn:bnd_first_2}
\end{align}
Finally, upon substitution from~\eqref{eqn:z_perp_rhs} and~\eqref{eqn:bnd_first_2} in~\eqref{eqn:zperp_1} we obtain that, for all $t \geq 0$,
\begin{align}
    \norm{g(t+1)} \leq \max \left\{ \mnorm{1 -  \frac{\delta \lambda_1}{\lambda_1 + \beta}}, \mnorm{1 - \frac{\delta \lambda_r}{\lambda_r + \beta}} \right\} \, \norm{g(t)} + \delta \lambda_1 \norm{\widetilde{K}(0)}_F \, \rho^{t+1} \norm{g(t)}. \label{eqn:g_t_norm_bnd_1}
\end{align}
Then,~\eqref{eqn:g_t_norm_bnd_1} and~\eqref{eqn:basic_ineq_2} prove part (i) of the theorem with 
$$\mu = \max \left\{ \mnorm{1 -  \frac{\delta \lambda_1}{\lambda_1 + \beta}}, \mnorm{1 - \frac{\delta \lambda_r}{\lambda_r + \beta}} \right\}.$$
Recall the definition of $\mu^*$ from~\eqref{eqn:opt_mu}. Note that~\cite[Chapter 11.3.3]{fessler2008image}
\begin{align}
    \mu \geq \frac{\frac{\lambda_1}{\lambda_1 + \beta}-\frac{ \lambda_r}{\lambda_r + \beta}}{\frac{\lambda_1}{\lambda_1 + \beta}+\frac{ \lambda_r}{\lambda_r + \beta}} = \frac{\beta \, (\lambda_1 - \lambda_r)}{ 2 \lambda_1 \lambda_r + \beta \, (\lambda_1 + \lambda_r)} = \mu^*, \label{eqn:bnd_mu}
\end{align}
where the equality $\mu = \mu^*$ holds true, if the value of $\delta$ is given by~\eqref{eqn:critical_delta}, which is
\begin{align*}
    \delta = \frac{2}{\frac{\lambda_1}{\lambda_1 + \beta} + \frac{\lambda_r}{\lambda_r + \beta}}.
\end{align*}
Note that as $\left(\frac{\lambda_r}{\lambda_r + \beta}\right) > 0$, 
\[\frac{2}{\frac{\lambda_1}{\lambda_1 + \beta} + \frac{\lambda_r}{\lambda_r + \beta}} < \frac{2}{\frac{\lambda_1}{\lambda_1 + \beta}} = 2 \left(\frac{\lambda_1 + \beta}{\lambda_1} \right).\]
Thus, the value of $\delta$ in~\eqref{eqn:critical_delta} satisfies the condition~\eqref{eqn:alpha_delta}.\\

{\bf Step III:} In this final step, we will prove part (ii) of the theorem. Note that, owing to the following fact, it suffices to show that the sequence of gradient norms $\left\{\norm{g(t)}\right\}_{t \geq 0}$ converges to zero.

\begin{fact} \label{fct:seq}
Consider an infinite sequence of non-negative values $\{s_t, ~ t = 0, \, 1, \ldots\}$ with 
\[s_{t} < s_{t-1}, ~ \forall t \geq 1, ~ \text{ and } ~\lim_{t \rightarrow \infty} s_t < L\]
where $L$ is a positive finite real number. Then, there exists $0 \leq T' < \infty$ such that $s_t < L, ~ \forall t > T'$.
\end{fact}

In part (i) of the theorem (see~\eqref{eqn:conv_1}), let  
\[\alpha_t = \left(\mu + \delta \lambda_1 \norm{K(0)-K_{\beta}}_F \, \rho^{t+1}\right).\]
Note that $\alpha_t \geq 0$ for all $t \geq 0$. Since $\rho < 1$, the sequence $\{\alpha_t\}_{t \geq 0}$ is strictly decreasing, which means, $\alpha_t < \alpha_{t - 1}, ~ \forall t \geq 1$, and $\lim_{t \rightarrow \infty}\alpha_t = \mu < 1$. 
Thus, due to Fact~\ref{fct:seq}, there exists a positive integer $\tau$ such that
\begin{align}
    \alpha_t < 1, \quad \forall t > \tau. \label{eqn:alpha_tau}
\end{align}
From the recursion~\eqref{eqn:conv_1} in part (i), we obtain that
\begin{align}
    & \norm{g(t+1)} \leq \left(\Pi_{k=\tau+1}^t \, \alpha_k\right) \norm{g(\tau + 1)}, ~~ \forall t > \tau. \label{eqn:zclm_1}
\end{align}
As $\{\alpha_t\}_{t \geq 0}$ is a strictly decreasing sequence,~\eqref{eqn:conv_1} also implies that
\begin{align}
    & \norm{g(\tau+1)} \leq \alpha_{\tau} \norm{g(\tau)} \leq \alpha_{0} \norm{g(\tau)}. \label{eqn:recur}
\end{align}
Using recursion $\tau$ times in~\eqref{eqn:recur} we get
\begin{align}
    \norm{g(\tau+1)} \leq \alpha_{0}^{\tau+1} \norm{g(0)}. \label{eqn:zclm_2}
\end{align}
Combining~\eqref{eqn:zclm_1} and~\eqref{eqn:zclm_2} we obtain that
\begin{align}
    \norm{g(t+1)} & \leq \left(\Pi_{k=\tau+1}^t \, \alpha_k\right) \alpha_{0}^{\tau+1} \norm{g(0)} 
    \leq \left(\Pi_{k=\tau+1}^t \, \alpha_{\tau +1}\right) \alpha_{0}^{\tau+1} \norm{g(0)} \nonumber \\
    & = \alpha_{\tau+1}^t \left(\dfrac{\alpha_{0}}{\alpha_{\tau+1}}\right)^{\tau+1}\norm{g(0)}, \, \forall t>\tau. \label{eqn:zclm_3}
\end{align}
Since $\alpha_{\tau+1} < 1$ (from~\eqref{eqn:alpha_tau}), we have $$\underset{t \rightarrow \infty}{\lim} \, \alpha_{\tau+1}^t \left(\dfrac{\alpha_{0}}{\alpha_{\tau+1}}\right)^{\tau+1} = 0.$$ From~\eqref{eqn:zclm_3} then it follows that
$$\underset{t \rightarrow \infty}{\lim} \, \norm{g(t+1)} = 0.$$
Part (ii) of the theorem follows from the definition of limit. \\

\newpage
\section{Experiments}
\label{sec:exp}

In this section, we present experimental results to demonstrate the theoretical convergence guarantees of Algorithm~\ref{algo_1} (see Section~\ref{sub:conv}), and its improvement over other state-of-the-art methods (see Section~\ref{sec:comp}). \\

{\bf Setup}: We conduct experiments for different collective data matrices $A$, chosen from the benchmark datasets available in SuiteSparse Matrix Collection\footnote{https://sparse.tamu.edu}. We consider four such datasets, namely \textit{``ash608''}, \textit{``bcsstm07''}, \textit{``gr\_30\_30''}, and \textit{``qc324''}. For example, for a particular dataset \textit{``ash608''}, the matrix $A$ has $608$ rows and $d=188$ columns.
The collective observation vector $B$ is such that $B=Ax^*$ where $x^*$ is a $188$ dimensional vector, all of whose entries are unity. To simulate the distributed server-agent architecture, the data points represented by the rows of the matrix $A$ and the corresponding observations represented by the elements of the vector $B$ are divided amongst $m=10$ agents numbered from $1$ to $10$. Since the matrix $A$ for this particular dataset has $608$ rows and $188$ columns, each of the nine agents $1, \ldots, \, 9$ has a data matrix $A^i$ of dimension $60 \times 188$ and a observation vector $B^i$ of dimension $60$. The pair $(A^{10}, \, B^{10})$ available to the tenth agent is of dimension $68\times188$ and $68$ respectively. The data points for the other three datasets are similarly distributed among $m=10$ agents.

As the matrix $\A$ is positive definite in each of these cases, the optimization problem~\eqref{eqn:opt_1} has a unique solution $x^*$ for all of these datasets. \\

\begin{figure*}[htb!]
\begin{subfigure}{.5\textwidth}
  \centering
  \includegraphics[width = \textwidth]{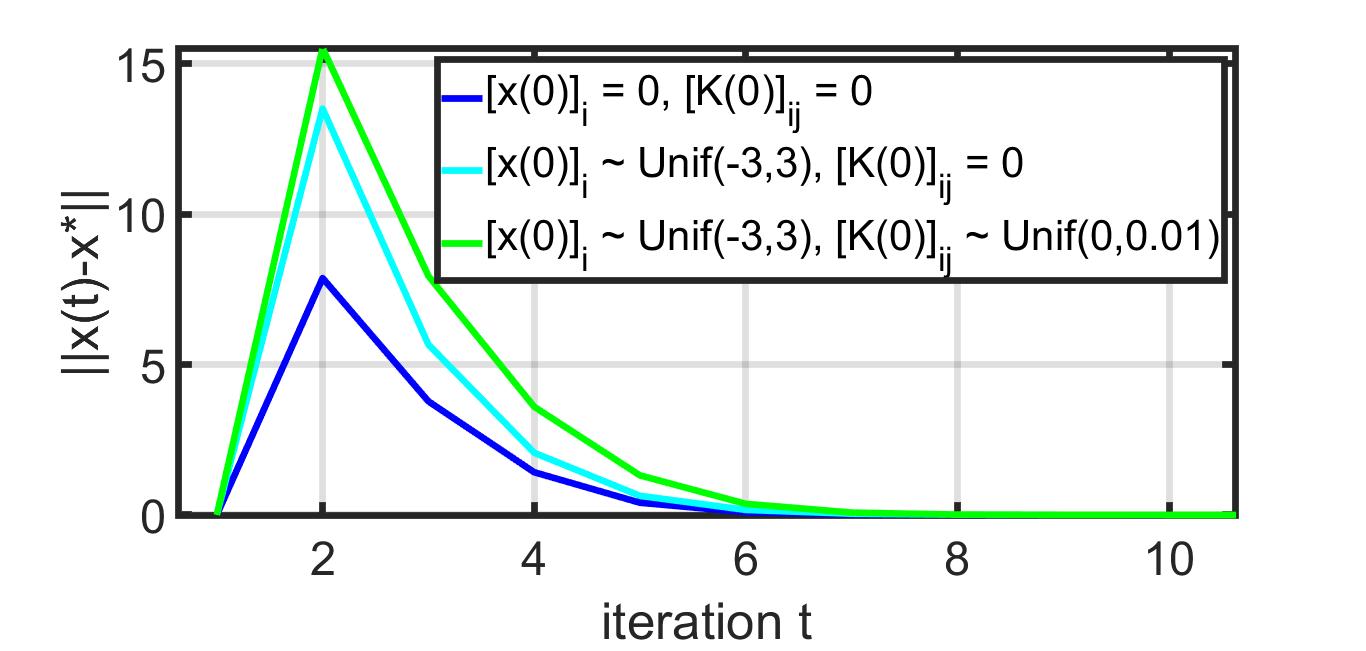}
  \caption{}
  \label{fig:ash_init}
\end{subfigure}%
\begin{subfigure}{.5\textwidth}
  \centering
  \includegraphics[width = \textwidth]{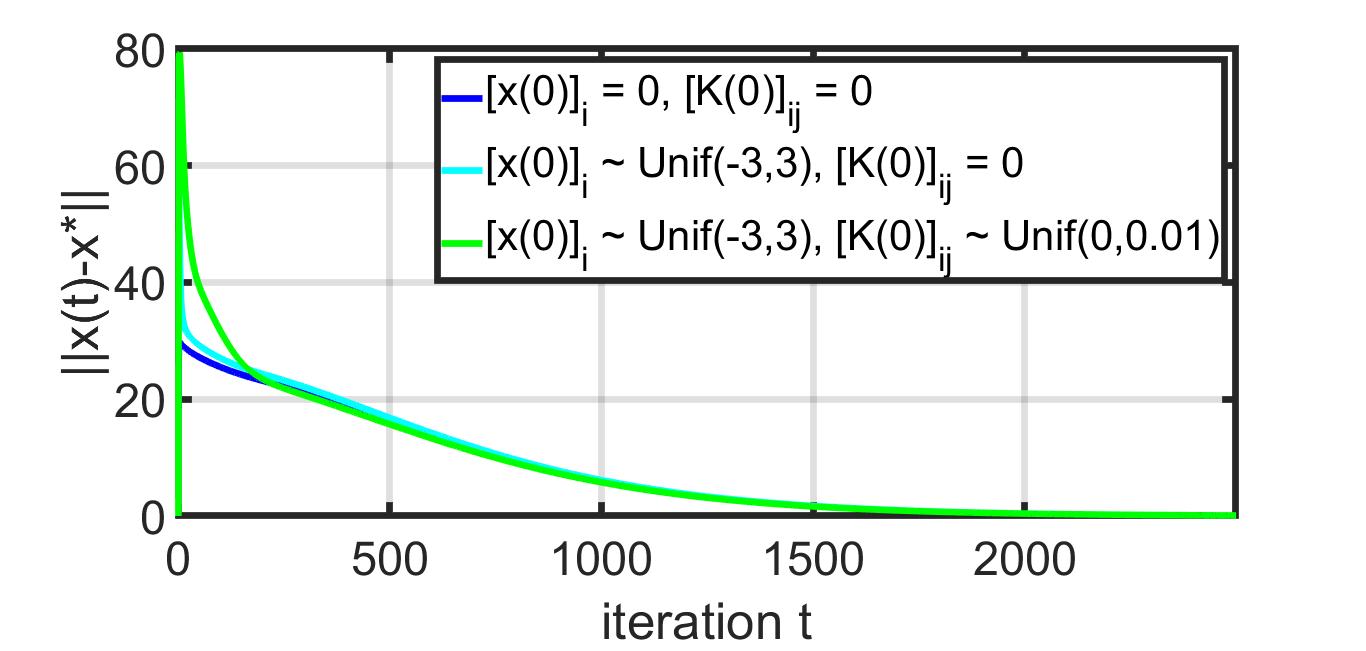}
  \caption{}
  \label{fig:gr_init}
\end{subfigure}
\caption{\footnotesize{Temporal evolution of error norm for estimate $\norm{x(t)-x^*}$ under Algorithm~\ref{algo_1} with different initialization; for the datasets (a) \textit{``ash608''} and (b) \textit{``gr\_30\_30''}.  (a) $\alpha = 0.1, \, \delta = 1, \, \beta = 0$; (b) $\alpha = 3 \times 10^{-3}, \, \delta = 0.4, \, \beta = 0$.}}
\label{fig:init}
\end{figure*}

\begin{figure*}[htb!]
\centering
\begin{subfigure}{.45\textwidth}
  \begin{center}
  \includegraphics[width = \textwidth]{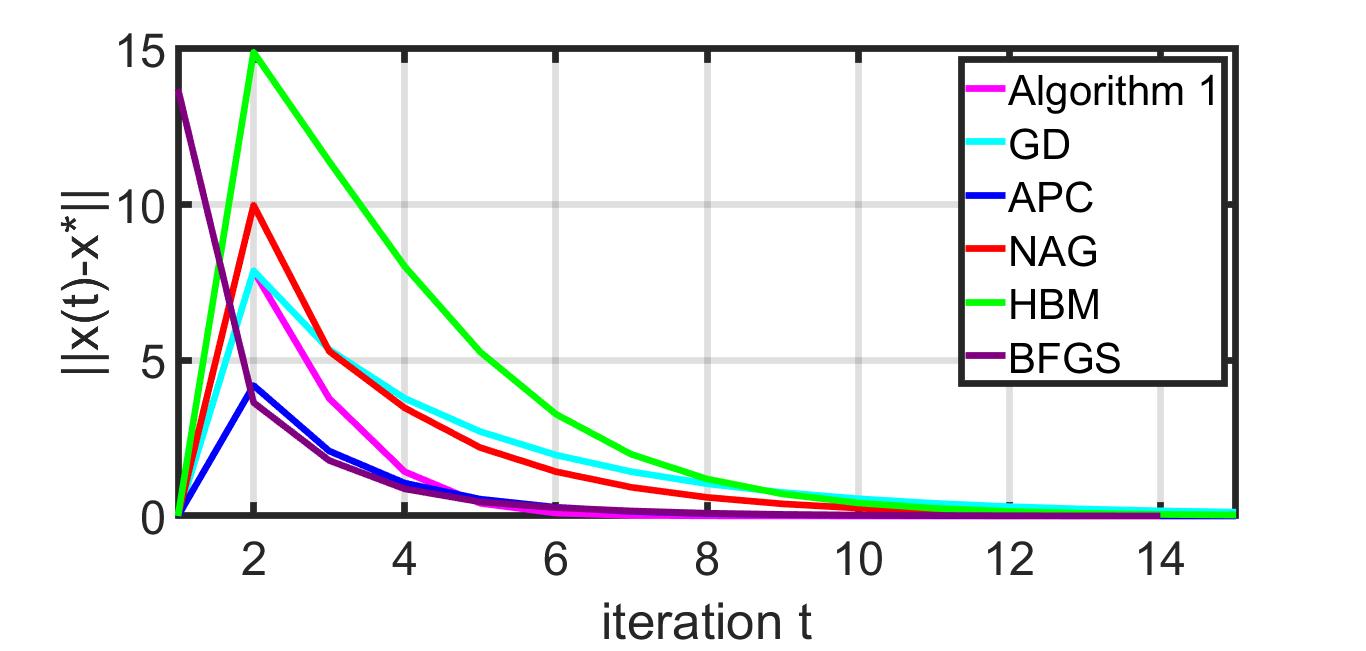}
  \caption{``ash608''}
  \end{center}
\end{subfigure}%
\begin{subfigure}{.45\textwidth}
  \begin{center}
  \includegraphics[width = \textwidth]{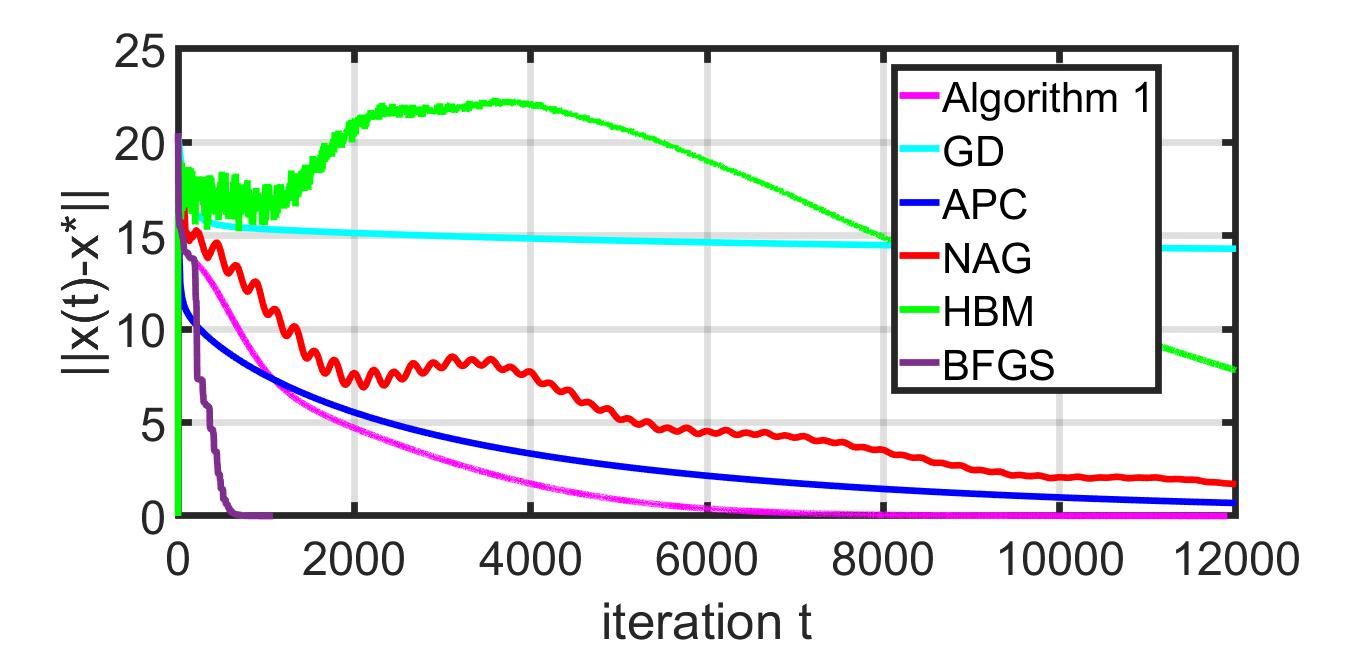}
  \caption{``bcsstm07''}
  \end{center}
\end{subfigure}
\begin{subfigure}{.45\textwidth}
  \begin{center}
  \includegraphics[width = \textwidth]{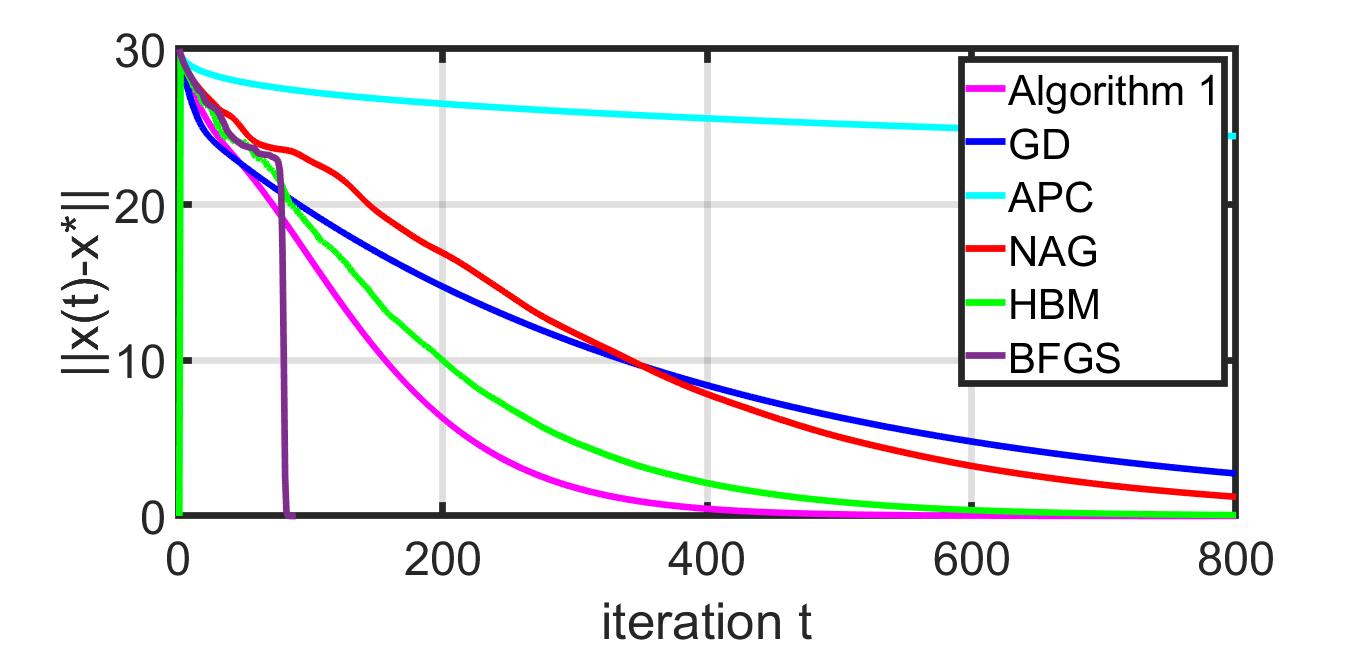}
  \caption{``gr\_30\_30''}
  \end{center}
\end{subfigure}%
\begin{subfigure}{.45\textwidth}
  \begin{center}
  \includegraphics[width = \textwidth]{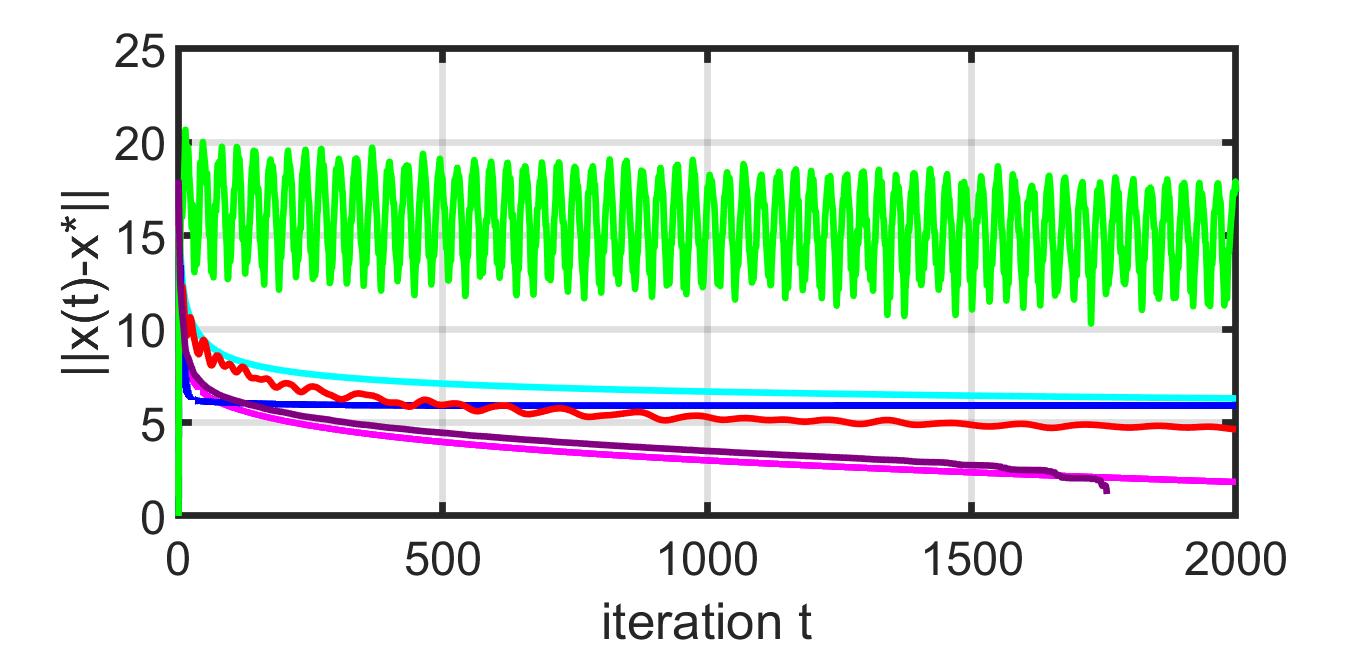}
  \caption{``qc324''}
  \end{center}
\end{subfigure}
\caption{Temporal evolution of error norm for estimate $\norm{x(t)-x^*}$, under Algorithm~\ref{algo_1}, GD, APC, NAG, HBM with optimal parameter choices and BFGS. Initialization for (a)-(d): (Algorithm~\ref{algo_1}) $x(0) = [0,\ldots,0]^T$, $K(0) = O_{d \times d}$; (GD, NAG, HBM) $x(0) = [0,\ldots,0]^T$; (APC) according to the algorithm; (BFGS) $x(0) = [0,\ldots,0]^T$, $M(0) = I$. The algorithms GD, NAG, HBM, and BFGS have been described in Section~\ref{sec:comp}. The APC algorithm can be found in~\cite{azizan2019distributed}.}
\label{fig:comp}
\end{figure*}

{\bf Global convergence of Algorithm~\ref{algo_1}}:
Since the solution set is a singleton, we apply Algorithm~\ref{algo_1} with $\beta = 0$ to solve the distributed least-squares problem~\eqref{eqn:opt_1} on the aforementioned datasets. To demonstrate the global nature of our algorithm, we simulate this algorithm with several choices for the initialization of the estimate $x(0)$ and the iterative pre-conditioner matrix $K(0)$ for two of these datasets, namely \textit{``ash608''} and \textit{``gr\_30\_30''}. The stepsize parameters $\alpha$ and $\delta$ have been chosen arbitrarily but sufficiently small. For these two datasets, we select the stepsize as:
\begin{enumerate}[label=$(\alph*)$]
    \item \textit{``ash608''}: $\alpha = 0.1, \, \delta = 1$;
    \item \textit{``gr\_30\_30''}: $\alpha = 3 \times 10^{-3}, \, \delta = 0.4$.
\end{enumerate}
With the respective stepsize as mentioned above, for either of these datasets we simulate Algorithm~\ref{algo_1} with three sets of initialization $(x(0),K(0))$: 
\begin{enumerate}[label=$(\alph*)$]
    \item each entry of $x(0)$ and $K(0)$ is zero;
    \item each entry of $x(0)$ is selected {\em uniformly} at random within $(-3,3)$ and each entry of $K(0)$ is zero;
    \item each entry of $x(0)$ and  $K(0)$ is selected {\em uniformly} at random within $(-3,3)$ and $(0,0.01)$ respectively.
\end{enumerate}
The simulation results for these two datasets are shown in Fig.~\ref{fig:init}. It can be seen that, the algorithm converges to $x^*$ irrespective of the initial choice of the entries in $x(0)$ and $K(0)$.\\

\begin{table}[h]
\caption{\it The optimal parameter values for different algorithms on real datasets.}
\begin{center}
\begin{tabular}{|p{1.4cm}|p{1.4cm}|p{1.7cm}|p{1.7cm}|p{1.7cm}|p{1.7cm}|}
\hline
Dataset & GD & NAG & HBM & APC & Algo.~\ref{algo_1} \\
\hline
ash608 & $\delta = 0.1163$ & $\delta = 0.08, \, \eta = 0.5$ & $\delta = 0.15, \, \eta = 0.29$ & $\gamma = 1.02, \, \beta =5.27$ & $\alpha = 0.1163, \, \delta = 1$\\
\hline
bcsstm07 & $\delta = 3 \times 10^{-7}$ & $\delta = 2 \times 10^{-7}, \, \eta =0.99$ & $\delta = 1 \times 10^{-7}, \, \eta =0.99$ & $\gamma = 1.09, \, \beta =12.8$ & $\alpha = 3 \times 10^{-7}, \, \delta =1$  \\
\hline
gr\_30\_30 & $\delta = 0.014$ & $\delta = 0.009, \, \eta =0.99$ & $\delta = 0.03, \, \eta =0.98$ & $\gamma = 1.09, \, \beta =12.8$ & $\alpha = 0.014, \, \delta =1$ \\
\hline
qc324 & $\delta = 0.85$ & $\delta = 0.57, \, \eta =0.99$ & $\delta = 0.03, \, \eta =0.98$ & $\gamma = 1.05, \, \beta =18.9$ & $\alpha = 0.85, \, \delta =1$ \\
\hline
\end{tabular}
\end{center}
\label{tab:parameters}
\end{table}

{\bf Comparison with related methods}:
The experimental results have been compared with the other algorithms in server-agent networks described in Section~\ref{sec:comp}, namely gradient descent method (GD) in a server-agent network, Nesterov's accelerated gradient-descent method (NAG) in a server-agent network, heavy-ball method (HBM) in a server-agent network, the accelerated projection consensus method (APC)~\cite{azizan2019distributed}, and the BFGS method in a server-agent network (ref. Fig.~\ref{fig:comp}). Among the other algorithms, APC has been recently proposed and speculated to be the fastest existing algorithm for solving~\eqref{eqn:opt_1}, if the minimum cost in~\eqref{eqn:opt_1} is zero. The distributed pre-conditioning scheme proposed in~\cite{azizan2019distributed} for improving the convergence rate of HBM has the same theoretical rate as APC.\\

The parameters for all of these algorithms are chosen such that the respective algorithms achieve their optimal (smallest) rate of convergence. For Algorithm~\ref{algo_1} with $\beta =0$, these optimal parameter values are given by $\alpha = \frac{2}{\lambda_1+\lambda_d}$ and $\delta = 1$. 
The optimal parameter expressions for the algorithms GD, NAG, and HBM can be found in~\cite{lessard2016analysis} and for that of APC in~\cite{azizan2019distributed}. We obtain these parameter values as listed in Table~\ref{tab:parameters}. The stepsize parameter for BFGS is selected using backtracking~\cite{kelley1999iterative}.
For each of the datasets, the initial estimate $x(0)$ has been chosen to be the $d$-dimensional zero vector for Algorithm~\ref{algo_1}, GD, NAG, HBM, and BFGS. The initial Hessian estimate $M(0)$ for the BFGS method has been chosen to be the identity matrix of dimension $d$. The initial pre-conditioner matrix $K(0)$ for the Algorithm~\ref{algo_1} is the zero matrix of dimension $d$. The initial $x(0)$ for the APC method is according to the algorithm. Note that evaluating the optimal tuning parameters for any of these algorithms requires knowledge about the smallest and largest eigenvalues of $\A$. 
\\

\begin{table}[h]
\caption{\it Comparisons between the number of iterations required by different algorithms with optimal parameter choices on real datasets to attain the specified values for the relative estimation errors $\epsilon_{tol} = \norm{x(t)-x^*}/\norm{x^*}$.}
\begin{center}
\begin{tabular}{|p{1.4cm}|p{1.4cm}|p{0.7cm}|p{0.6cm}|p{1.2cm}|p{1.2cm}|p{1.2cm}|p{1.2cm}|p{1.1cm}|}
\hline
Dataset & $\kappa(\A)$ & $\epsilon_{tol}$ & GD & NAG & HBM & APC & BFGS & Algo.~\ref{algo_1} \\
\hline
ash608 & $11.38$ & $10^{-4}$ & $37$ & $23$ & $21$ & $15$ & $15$ & $9$ \\
\hline
bcsstm07 & $5.8 \times 10^{7}$ & $10^{-4}$ & $> 10^5$ & $5.64 \times 10^4$ & $4.87 \times 10^4$ & $4.85 \times 10^4$ & $877$ & $1.19 \times 10^4$ \\
\hline
gr\_30\_30 & $3.79 \times 10^{4}$ & $10^{-4}$ & $>10^5$ & $1.94 \times 10^3$ & $1.13 \times 10^3$ & $1.11 \times 10^3$ & $85$ & $7.42 \times 10^2$  \\
\hline
qc324 & $2.15 \times 10^{9}$ & $0.1$ & $>10^5$ & $2.83 \times 10^4$ & $4.41 \times 10^4$ & $>10^5$ & $1.74 \times 10^3$ & $1.94 \times 10^3$  \\
\hline
\end{tabular}
\end{center}
\label{tab:time_comp}
\end{table}

We compare the number of iterations needed by these algorithms to reach a {\em relative estimation error} defined as 
$$\epsilon_{tol} = \frac{\norm{x(t)-x^*}}{\norm{x^*}}.$$
The algorithm parameters have been set such that the respective algorithms will have their smallest possible convergence rates. Their specific values have been mentioned in the previous paragraph.
Clearly, Algorithm~\ref{algo_1} performs fastest among the algorithms except BFGS, significantly for the datasets \textit{``bcsstm07''} and \textit{``gr\_30\_30''} (ref. Table~\ref{tab:time_comp}). When the condition number of $\A$ is too small (\textit{``ash608''}) or too large (\textit{``qc324''}), the difference is not significant but still Algorithm~\ref{algo_1} is faster than the other methods except BFGS. As the matrices $\A$ satisfies the condition of Corollary~\ref{cor:super_conv}, the rate of convergence of Algorithm~\ref{algo_1} is zero for each of the four datasets. Whereas, the algorithms GD, NAG, HBM, and APC are known to have a linear rate of convergence (ref. Section~\ref{sec:comp}), which means, their rate of convergence is positive. Thus, our theoretical claim on improvements over these methods is corroborated by the above experimental results. \\

\begin{figure*}[htb!]
\begin{subfigure}{.5\textwidth}
  \centering
  \includegraphics[width = \textwidth]{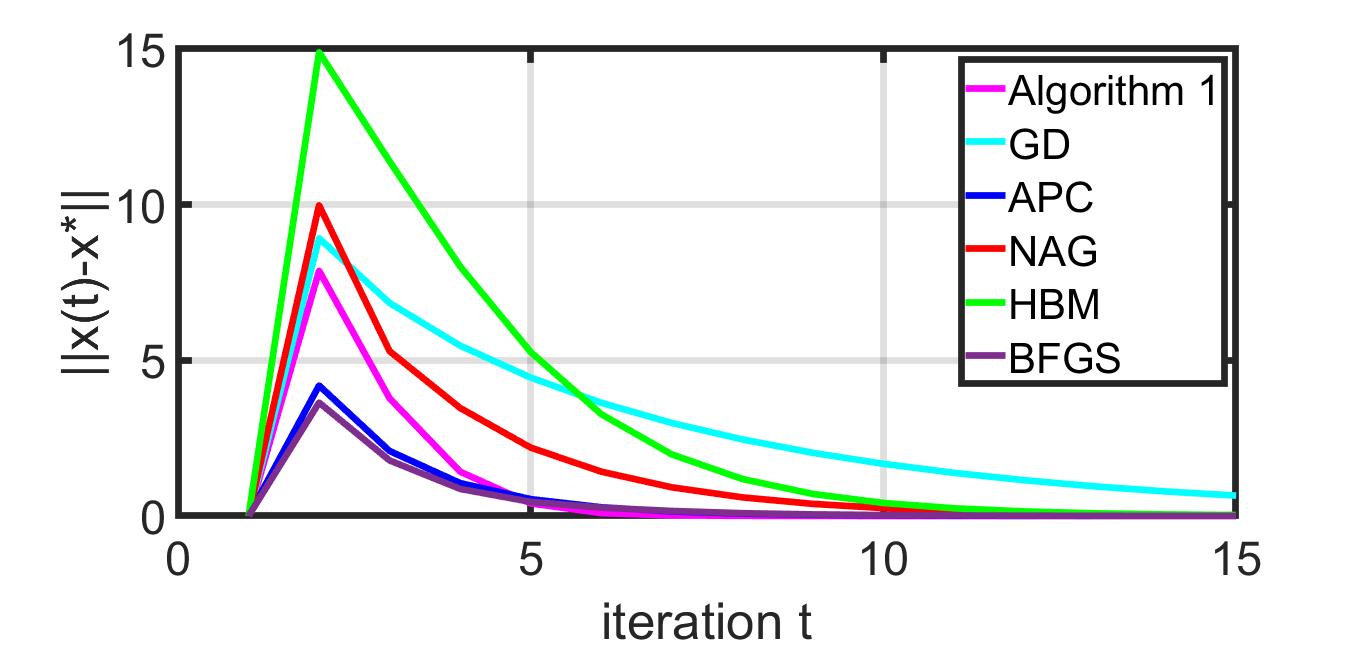}
  \caption{}
  \label{fig:ash_noise}
\end{subfigure}%
\begin{subfigure}{.5\textwidth}
  \centering
  \includegraphics[width = \textwidth]{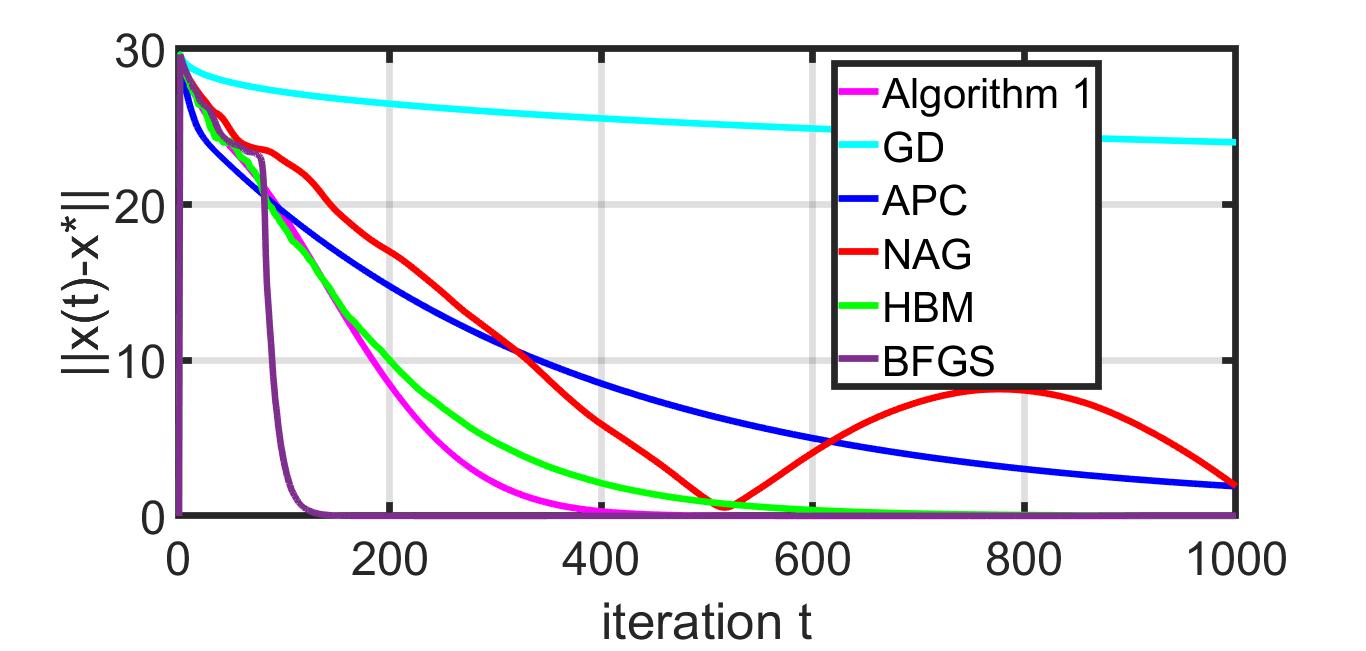}
  \caption{}
  \label{fig:gr_noise}
\end{subfigure}
\caption{\footnotesize{Temporal evolution of error norm for estimate $\norm{x(t)-x^*}$ in presence of system noise, under Algorithm~\ref{algo_1}, GD, APC, NAG, HBM with optimal parameter choices and BFGS; for the datasets (a) \textit{``ash608''} and (b) \textit{``gr\_30\_30''}. Initialization for (a) and (b) both: (Algorithm~\ref{algo_1}) $x(0) = [0,\ldots,0]^T$, $K(0) = O_{d \times d}$; (GD, NAG, HBM) $x(0) = [0,\ldots,0]^T$; (APC) according to the algorithm; (BFGS) $x(0) = [0,\ldots,0]^T$, $M(0) = I$. The algorithms GD, NAG, HBM, and BFGS have been described in Section~\ref{sec:comp}. The APC algorithm can be found in~\cite{azizan2019distributed}.}}
\label{fig:noise}
\end{figure*}

\begin{table}[h]
\caption{\it Comparisons between the asymptotic estimation errors $\lim_{t \rightarrow \infty} \norm{x(t)-x^*}$ for different algorithms with optimal parameter choices on real datasets. For Algorithm~\ref{algo_1}, $K(0) = O_{d \times d}$.}
\begin{center}
\begin{tabular}{|p{1.5cm}|p{1.5cm}|p{1.5cm}|p{1.5cm}|p{1.5cm}|p{1.5cm}|p{1.5cm}|p{1.2cm}|}
\hline
Dataset & $w$ & GD & NAG & HBM & APC & BFGS & Algo.~\ref{algo_1} \\
\hline
ash608 & $6.81 \times 10^{-4}$ & $3.46 \times 10^{-4}$ & $9.21 \times 10^{-4}$ & $10^{-4}$ & $3.74 \times 10^{-4}$ & $\infty$ & $0$ \\
\hline
gr\_30\_30 & $1.5 \times 10^{-3}$ & $7.68$ & $1.86$ & $8.5 \times 10^{-3}$ & $0.45$ & $1.49 \times 10^{-2}$ & $0$ \\
\hline
\end{tabular}
\end{center}
\label{tab:sse}
\end{table}

\subsection{Effect of Noisy Computation}

We consider the same distributed least-squares problem as above, but instead of ideal machines, the algorithms are implemented in the presence of system noise. 
Specifically, for the datasets \textit{``ash608''} and \textit{``gr\_30\_30''}, we simulate the algorithms by adding system noise to the iterated variables. For the algorithms GD, NAG, HBM, and Algorithm~\ref{algo_1}, the system noise has been generated in the form of rounding-off each entry of all the iterated variables in the respective algorithms to {\em four} decimal places. For an unbiased comparison between all the algorithms, we would like to have approximately the same level of noise $w$ for all algorithms. As done for the algorithms GD, NAG, HBM, and Algorithm~\ref{algo_1}, the rounding-off process does not generate a similar value of $w$ for the APC and BFGS algorithms. For APC, instead of rounding-off the entries, we add uniformly distributed random numbers in the range $(0,10^{-6})$ for both the datasets. Similarly for BFGS, we add uniformly distributed random numbers in the range $(0,5 \times 10^{-6})$ and $(0,2 \times 10^{-6})$ respectively for the datasets \textit{``ash608''} and \textit{``gr\_30\_30''}.\\

We compare the {\em asymptotic estimation error}, defined as $\lim_{t \rightarrow \infty} \norm{x(t)-x^*}$, of these algorithms on both of the datasets. The asymptotic estimation error is measured by waiting until the error norm $\norm{x(t)-x^*}$ does not change anymore with $t$. The algorithm parameters have been set at their optimal values. The exact values of all the parameters, including the initialization of the variables, have been provided earlier in this section. The temporal evolution of the error norm in the estimate has been plotted in Fig.~\ref{fig:noise} for both the datasets. We observe the asymptotic error of Algorithm~\ref{algo_1} to be less compared to the other algorithms (ref. Table~\ref{tab:sse}). The asymptotic error for BFGS on the dataset \textit{``ash608''} grows unbounded after $360$ iterations. The approximate Hessian matrix in the BFGS method needs to be non-singular at every iteration. Nevertheless, this condition is violated in the presence of noise, which results in the growth of error. 
 \\

\newpage
\section{Summary}
\label{sec:summary}

We have considered the multi-agent linear least-squares problem in a server-agent network. Several algorithms are available for solving this problem without hampering the privacy of each agent's raw data. However, all these methods' convergence speed is fundamentally limited by the condition number of the collective data points. We have proposed an iterative pre-conditioning technique that is robust to the condition number. Thus, we can reach a satisfactory neighborhood of the desired solution in a provably fewer number of iterations compared to the existing state-of-the-art algorithms. The theoretical analysis for convergence of our algorithm and its comparison with related methods have been supported through experiments on real-world datasets, in ideal and noisy computational environments. 
 
\section*{Acknowledgements}

This work is being carried out as a part of the Pipeline System Integrity Management Project, which is supported by the Petroleum Institute, Khalifa University of Science and Technology, Abu Dhabi, UAE. Nirupam Gupta was sponsored by the Army Research Laboratory under Cooperative Agreement W911NF- 17-2-0196.



\newpage
\bibliographystyle{unsrt}
\bibliography{refs}

\newpage
\appendix
\section{Other Proofs}
\label{sec:proof}



\subsection{Proof of Corollary~\ref{cor:super_conv}}
\label{prf:super_conv}
Note that the solution $x^*$, defined by~\eqref{eqn:opt_1}, is unique if and only if the matrix $\A$ is full-rank, which means, $r=d$. 
Thus, in this particular case, $\A$ is symmetric positive definite matrix and has positive real eigenvalues. In other words, we have $\lambda_1 \geq \, \ldots \, \geq \lambda_d > 0$. Moreover, the inverse matrix $\left(\A\right)^{-1}$ exists, and is also a symmetric positive definite matrix. \\

As $r = d$, substituting $\beta = 0$ in~\eqref{eqn:opt_mu} and~\eqref{eqn:opt_rho}, respectively, we obtain that $\mu^* = 0$ and 
\[\varrho = \frac{\lambda_1 - \lambda_d}{\lambda_1 + \lambda_d}.\]
Now, part (i) of Theorem~\ref{thm:thm1} implies that, for $\delta = 1$ obtained by substituting $\beta = 0$ in~\eqref{eqn:critical_delta}, 
\[\norm{g(t+1)} ~ \leq ~ \lambda_1 \norm{K(0)-K_{\beta}}_F \, \rho^{t+1} \,  \norm{g(t)}, \quad \forall t \geq 0,\]
where $\varrho \leq \rho < 1$. 

\subsection{Proof of Lemma~\ref{lem:gd}}
\label{prf:gd}

Comparing the update equations of the gradient-descent method in server-agent networks, in~\eqref{eqn:dgd}, and that of Algorithm~\ref{algo_1} in~\eqref{eqn:x_update}, we see that Algorithm~\ref{algo_1} with $K(t) = I \quad \forall t\geq 0$ is the gradient-descent method in server-agent networks.
Thus, for the gradient-descent method in server-agent networks we define $K_{\beta} = I$ to which the sequence of matrices $\{K(t)\}$ converges. Now recall the definition of $\widetilde{K}(t)$ from~\eqref{eqn:tilde_k}. For the gradient-descent algorithm, we then have 
\begin{align}
    \widetilde{K}(t) = 0 \quad \forall t \geq 0.  \label{eqn:kt_zero}
\end{align}

Now we proceed exactly as the proof of Theorem~\ref{thm:thm1}, and arrive at~\eqref{eqn:zperp_1} with $\widetilde{K}(t) = 0 \quad \forall t \geq 0$ and $K_{\beta} = I$. In other words,
\begin{align*}
    \norm{g(t+1)} \leq \norm{\left(I -  \delta \A Q \, \right) \, g(t)} \quad \forall t \geq 0.
\end{align*}
Substituting the eigen-expansion of~\eqref{eqn:trunc_A} in the above inequality, we get
\begin{align}
    \norm{g(t+1)} \leq \norm{ V \, Diag\left(\left(1-\delta \lambda_1\right), \ldots, \, \left(1-\delta \lambda_r \right), \, \underbrace{1, \ldots, 1}_{d-r}\right) \, V^T \, g(t)} \quad \forall t \geq 0. \label{eqn:gt_gd}
\end{align}
Following the argument after~\eqref{eqn:all_matrices_g_t}, if $\delta \in \left(0, \frac{2}{\lambda_1}\right)$ we have from~\eqref{eqn:gt_gd}:
\begin{align*}
    \norm{g(t+1)} \leq \max \left\{ \mnorm{1 - \delta \lambda_1},  \mnorm{1 - \delta \lambda_r} \right\} \, \norm{g(t)} \quad \forall t \geq 0,
\end{align*}
and 
\[\mnorm{1 - \delta \lambda_i} < 1, \, i=1,\ldots,r.\]
Defining $\max \left\{ \mnorm{1 - \delta \lambda_1}, \mnorm{1 - \delta \lambda_r} \right\} = \mu$ we have derived~\eqref{eqn:gd_rate} in the statement of this lemma.
The smallest possible $\mu$ is given by 
\[\mu \geq \frac{\lambda_1-\lambda_r}{\lambda_1+\lambda_r},\]
which is $\mu_{GD}$ in~\eqref{eqn:mu_def}. Thus, the proof of the lemma is complete.

\subsection{Proof of Theorem~\ref{thm:comp}}
\label{prf:comp}

The statement of this theorem is a direct application of a general result stated in the following lemma. Lemma~\ref{lem:tau} compares the convergence rate of two algorithms for solving~\eqref{eqn:opt_1}, both of them having time-varying rates of contraction, one of which is smaller than the other after a finite number of iterations. We refer these algorithms as Algorithm-I and Algorithm-II.\\

To present this claim, we define a few notations. 
\begin{itemize}
\setlength\itemsep{0.5em}
    \item Let the gradient computed by Algorithm-I and Algorithm-II be denoted by $g_1(t)$ and $g_2(t)$, respectively, after $t$ iterations. 
    \item The least upper-bound of these gradients are known as follows:
\begin{align}
    \norm{g_1(t+1)} & \leq \alpha_t \norm{g_1(t)} \, \text{and} \,
    \norm{g_2(t+1)} \leq \eta_t \norm{g_2(t)}, \, \forall t. \label{eqn:z1} 
\end{align}
    \item Define the ratio between the instantaneous convergence rates as
    \begin{align}
        r_t := \dfrac{\alpha_t}{\eta_t}, \, \forall t. \label{eqn:rt}
    \end{align}
\end{itemize}

\begin{lemma} \label{lem:tau}
Suppose $\{\alpha_t > 0\}_{t \geq 0}$ is a strictly decreasing sequence and there exists some positive integer $\tau$ such that
\begin{align}
    \alpha_t < \eta_t < 1 \, \forall t>\tau. \label{eqn:alpha_eta}
\end{align}
Then $\Bar{r}:=\underset{t>\tau}{\max} \, r_t < 1$ and there exists a positive number $c$ such that
\begin{align}
    \norm{g_1(t+1)} \leq c \left(\Bar{r}\right)^{t+1} \, \left(\Pi_{k=0}^{t} \, \eta_k\right) \, \norm{g_1(0)}, \, \forall t > \tau. \label{eqn:z2}
\end{align}
\end{lemma}

\begin{proof}[Proof of Lemma~\ref{lem:tau}]
From the recursion~\eqref{eqn:z1}, we get
\begin{align}
    & \norm{g_1(t+1)} \leq \left(\Pi_{k=\tau+1}^t \, \alpha_k\right) \norm{g_1(\tau + 1)}, \, \forall t > \tau. \label{eqn:z3}
\end{align}
Again from~\eqref{eqn:z1} and the fact that $\{\alpha_t > 0\}_{t \geq 0}$ is a strictly decreasing sequence, we have
\begin{align}
    & \norm{g_1(\tau+1)} \leq \alpha_{\tau} \norm{g_1(\tau)} \leq \alpha_{0} \norm{g_1(\tau)}. \label{eqn:g3}
\end{align}
Using recursion $\tau$ times in~\eqref{eqn:g3} we get
\begin{align}
    \norm{g_1(\tau+1)} \leq \alpha_{0}^{\tau+1} \norm{g_1(0)}. \label{eqn:z4}
\end{align}
Combining~\eqref{eqn:z3} and~\eqref{eqn:z4}, we have
\begin{align*}
    \norm{g_1(t+1)} & \leq \left(\Pi_{k=\tau+1}^t \, \alpha_k\right) \alpha_{0}^{\tau+1} \norm{g_1(0)}, \, \forall t>\tau.
\end{align*}
Upon substituting from the definition of $r_t$ in~\eqref{eqn:rt} we get
\begin{align*}
    \norm{g_1(t+1)} & \leq \left(\Pi_{k=\tau+1}^t \, r_k \eta_k\right) \alpha_{0}^{\tau+1} \norm{g_1(0)}
    = \left(\Pi_{k=\tau+1}^t \, r_k \eta_k\right) \left(\Pi_{k=0}^{\tau} \, \eta_k\right) \left(\dfrac{\alpha_{0}^{\tau+1}}{\Pi_{k=0}^{\tau} \, \eta_k}\right) \norm{g_1(0)}.
\end{align*}    
Using the definition $\Bar{r} = \underset{t>\tau}{\max} \, r_t$ we get
\begin{align*}
    \norm{g_1(t+1)} & \leq \left(\Pi_{k=\tau+1}^t \, \Bar{r} \eta_k\right) \left(\Pi_{k=0}^{\tau} \, \eta_k\right) \left(\dfrac{\alpha_{0}^{\tau+1}}{\Pi_{k=0}^{\tau} \, \eta_k}\right) \norm{g_1(0)}.
\end{align*}    
Algebraic simplification on the R.H.S. of the inequality leads to
\begin{align*}
    \norm{g_1(t+1)} 
    & \leq \left(\Bar{r}\right)^{t+1} \left(\Pi_{k=0}^{t} \, \eta_k\right) \, \left(\dfrac{\alpha_{0}}{\Bar{r}}\right)^{\tau+1} \left(\dfrac{1}{\Pi_{k=0}^{\tau} \, \eta_k}\right) \, \norm{g_1(0)}, \, \forall t>\tau.
\end{align*}
Thus, we have derived~\eqref{eqn:z2} with $c = \left(\frac{\alpha_{0}}{\Bar{r}}\right)^{\tau+1} \left(\frac{1}{\Pi_{k=0}^{\tau} \, \eta_k}\right)$.
Equation~\eqref{eqn:rt} and~\eqref{eqn:alpha_eta} together implies that
\[r_t  < 1 \, \forall t>\tau.\] Thus, $\Bar{r} = \underset{t>\tau}{\max} \, r_t < 1$ and the proof is complete.
\end{proof}

Now we apply this lemma for proving the theorem. For this, we need to show that the condition~\eqref{eqn:alpha_eta} of Lemma~\ref{lem:tau} is satisfied, which proceeds as follows.\\ 

Consider Algorithm-I and Algorithm-II in Lemma~\ref{lem:tau} respectively to be Algorithm~\ref{algo_1} and the gradient-descent algorithm in server-agent networks. From~\eqref{eqn:conv_1} in Theorem~\ref{thm:thm1}, for some $\delta > 0$ we have \[\alpha_t := \left(\mu^* + \delta \lambda_1 \norm{K(0)-K_{\beta}}_F \, \rho^{t+1}\right).\] 
Similarly, from~\eqref{eqn:gd_rate} in Lemma~\ref{lem:gd}, for some $\delta > 0$ we have
\[\eta_t= \mu_{GD}.\]
As $\beta > 0$ and $\lambda_1>\lambda_r$, from~\eqref{eqn:opt_mu} and~\eqref{eqn:mu_def} we can see that $\mu^* < \mu_{GD}$.
Since $\rho < 1$, the sequence $\{\alpha_t > 0\}_{t \geq 0}$ is strictly decreasing and $\lim_{t \rightarrow \infty}\alpha_t = \mu^*$.
Hence, we have a strictly decreasing sequence $\{\alpha_t > 0\}_{t \geq 0}$ such that $\lim_{t \rightarrow \infty}\alpha_t < \mu_{GD}$. Thus, the conditions of Fact~\ref{fct:seq} hold, and we have a positive integer $\tau$ such that
\begin{align}
    \alpha_t < \mu_{GD} = \eta_t \quad \forall t > \tau. \label{eqn:alpha_seq}
\end{align}
Thus, the condition~\eqref{eqn:alpha_eta} of Lemma~\ref{lem:tau} is satisfied. Then,~\eqref{eqn:z2} holds with some positive quantities $c$ and $\Bar{r} < 1$.
Finally, substituting $\eta_t = \mu_{GD}$ in~\eqref{eqn:z2} we get~\eqref{eqn:z2_lem1} with $r = \Bar{r}$. Since $\Bar{r} < 1$, the proof of the theorem is complete.

\subsection{Proof of Lemma~\ref{lem:k}}
\label{prf:k}

Observing that $A^T A = \sum_{i=1}^m (A^i)^T A^i$ and $A^T B = \sum_{i=1}^m (A^i)^T B^i$, from~\eqref{eqn:Rij} we have
\[ \sum_{i=1}^m R^i_j(t) = \left[\left(\A+\beta I\right)k_j(t) - e_j\right]. \]
Upon substituting from above, dynamics~\eqref{eqn:kcol_update} can be rewritten as
\begin{align} \label{eqn:kcol_2}
    k_j(t+1) & = k_j(t) - \alpha \left[\left(\A+\beta I\right)k_j(t) - e_j\right] \quad j=1,\ldots,d, \quad \forall t \geq 0.
\end{align}
For each iteration $t$, define $\widetilde{k}_j(t) = k_j(t) - k_{j\beta}$. Recall, from the definition~\eqref{eqn:def_k_beta}, that $K_{\beta} = \left(\A + \beta I\right)^{-1}$. Then for each column $j=1,\ldots,d$ of $K_{\beta}$ we have
\begin{align*}
    \left(\A+\beta I\right) k_{j\beta} = e_j.
\end{align*}
Upon substituting in~\eqref{eqn:kcol_2} and using the definition of $\widetilde{k}_j(t)$,
\begin{align}
    \widetilde{k}_j(t+1) = \left[I- \alpha \left(\A+\beta I\right)\right] \widetilde{k}_j(t).
\end{align}
Since $(\A+\beta I)$ is positive definite for $\beta >0$, there exists $\alpha \in \left(0,\frac{2}{\lambda_1+\beta}\right)$ for which there is a positive $\rho < 1$ such that (ref. Corollary 11.3.3 of~\cite{fessler2008image})
\begin{align}
    \norm{\widetilde{k}_j(t+1)} \leq \rho \norm{\widetilde{k}_j(t)} \quad j=1,\ldots,d, \quad \forall t \geq 0. \label{eqn_kconv}
\end{align}
Recall that the {\em condition number} of a symmetric positive definite matrix, which we denote by $\kappa(\cdot)$, is equal to the ratio between the matrix's largest and the smallest eigenvalues~\cite{fessler2008image}.
Then, the smallest value of $\rho$ is given by (ref. Chapter 11.3.3 of~\cite{fessler2008image})
\begin{align}
    \rho \geq \frac{\kappa\left(\A+\beta I\right)-1}{\kappa\left(\A+\beta I\right)+1}. \label{eqn_kcond_num}
\end{align}
As the maximum and minimum eigenvalues of $(\A+\beta I)$ respectively are $\left(\lambda_1+\beta\right)$ and $\left(\lambda_d+\beta\right)$, we have $\kappa\left(\A+\beta I\right)=\dfrac{\lambda_1+\beta}{\lambda_d+\beta}$. Thus, the lower bound in~\eqref{eqn_kcond_num} simplifies to $\varrho$ in~\eqref{eqn:opt_rho}. The claim follows from~\eqref{eqn_kconv} and~\eqref{eqn_kcond_num}.

\newpage

\section{Robustness against System Noise}
\label{sec:noise}

In this appendix, we analyze the convergence of Algorithm~\ref{algo_1} in the presence of system noise. In practice, the source of system noise can be finite precision of the machines~\cite{holi1993finite} or error in quantization~\cite{gold1966effects}. Low-precision data representation and computation has gained significant attention of researchers in machine learning algorithms~\cite{perez2014multilayer, gupta2015deep, lesser2011effects}.\\

The system noise is modeled as follows. We will denote the actual values with a superscript `$o$' to distinguish them from their noisy counterpart.

\begin{itemize}
    \item For iteration $t = 0,1,...$ and $j = 1,...,d$, 
    \begin{align}
        k_j(t) = k^o_j(t) + w^k_j(t), \label{eqn:noisy_k}
    \end{align}
    where $w^k_j(t) \in \R^d$ is an additive noise vector, and  $k^o_j(t)$ is the actual value of $k_j(t)$ when the noise vector $w^k_j(t)$ is zero.
    \item For iteration $t = 0,1,...$,
    \begin{align}
        x(t) = x^o(t) + w^x(t), \label{eqn:noisy_x}
    \end{align}
    where $w^x(t) \in \R^d$ is an additive noise vector, and  $x^o(t)$ is the actual value of $x(t)$ when the noise vector $w^x(t)$ is zero.
    \item The noise vectors are upper bounded in norm:
     \begin{align}
        \norm{w^k_j(t)} \leq w, \, \norm{w^x(t)} \leq w, \quad j\in \{1, \, \ldots, \, d\}, \, \forall t \geq 0, \label{eqn:noise_bd}
    \end{align}
    for some $w>0$.
\end{itemize}
Note that, no assumption about the probability distribution of the noise have been made. \\

The presence of system noise affects the convergence of iterative algorithms. Not only the system noise increases the error at any iteration, but also, this error propagates over the subsequent iterations, thereby resulting in a more significant error due to accumulation~\cite{holi1993finite}. Hence, it is essential to analyze the total accumulated error of an iterative method in the presence of system noise. The following result provides an upper bound on the asymptotic error of the proposed algorithm.
To be able to present the result of this section, we define a few notations. Define
\begin{align*}
    \Tilde{k}^0_j(t) & = k^0_j(t) - k_{j\beta}, \, j=1,...,d, \, \forall t \\
    S(t)^2 & = \sum_{j=1}^d \left( \rho^{t} \norm{\Tilde{k}^0_j(0)} + \left(1+\rho+...+\rho^{t} \right)w \right)^2, \, \forall t \\
    R(t) & = \lambda_1 S(t), \, \forall t \\
    \rho_{j} & = \dfrac{\norm{\Tilde{k}_j^o(0)}}{\norm{\Tilde{k}_j^o(0)}+w}, \, j=1,...,d, \\
    w_{bd} & = \dfrac{\left(1-\rho \right)}{
    \lambda_1 \sqrt{d}},
\end{align*}
where $k_{j\beta}$ is the $j$-th column of $K_{\beta}$ for $j=1,\ldots,d$.

Let $x^*$ be a point of minima of~\eqref{eqn:opt_1} and $x(t)$ be the estimate at iteration $t$ of Algorithm~\ref{algo_1}. Define the estimation error for iteration $t$ as
\begin{align}
    z(t) = x(t) - x^*. \label{eqn:err}
\end{align}
Upon substituting the definition~\eqref{eqn:noisy_x} in~\eqref{eqn:err} we get that
\begin{align*}
    z(t) = x^o(t) + w^x(t) -x^*.
\end{align*}
Noting that the actual value of noisy $z(t)$, which we denote by $z^0(t)$, is given by $x^0(t) - x^*$, we have
\begin{align}
    z(t) = z^o(t) + w^x(t), \quad \forall t \geq 0. \label{eqn:noisy_z}
\end{align}
We then have the following theorem.

\begin{theorem} \label{thm:noise}
Consider Algorithm~\ref{algo_1} with $\beta = 0$. 
Suppose the set of minimums $X^*$, defined in~\eqref{eqn:opt_1} is singleton, $0 < \alpha < \frac{2}{\lambda_1 + \beta}$ and the following conditions hold:
\begin{align}
    \rho & < \rho_{j}, \, j=1,...,d, \label{eqn:assump_2a} \\
    w & < w_{bd}. \label{eqn:assump_2b}
\end{align}
Then, for $\delta = 1$ 
\begin{itemize}
    \item there exists $T' < \infty$ such that $R(T'+1) < 1$ and
\begin{align}
    \lim_{t \rightarrow \infty} \norm{z(t)} < \dfrac{w}{1 - R(T'+1)}, \label{eqn:sse}
\end{align}
    \item with
    \begin{align}
        \lim_{t \rightarrow \infty} \dfrac{w}{1 - R(t)} = \dfrac{w}{1 - w/w_{bd}}. \label{eqn:sse_lim}
    \end{align}
\end{itemize}
\end{theorem}

Theorem~\ref{thm:noise} provides an upper bound on the asymptotic value of the estimation error norm of Algorithm~\ref{algo_1} if the noise level $w$ is sufficiently small. This condition~\eqref{eqn:assump_2b} is likely to hold if the condition number of the matrix $\A$ is not very large. Even if this condition does not hold, Algorithm~\ref{algo_1} can still result in lower asymptotic error than the related algorithms mentioned in Section~\ref{sec:comp}. This has been demonstrated by experiments in Section~\ref{sec:exp}.

\begin{proof}[Proof of Theorem~\ref{thm:noise}]
The proof is divided into two steps. \\

{\bf Step I:} Recall that, $k_{j\beta}$ denote the $j$-th column of matrix $K_{\beta}$ where $j=1, \ldots, \, d$.
Due to noise,~\eqref{eqn:kcol_update} which has been shown equivalent to~\eqref{eqn:kcol_2} becomes
\begin{align*}
    k^0_j(t+1) & = k_j(t) - \alpha \left[\left(\A+\beta I\right)k_j(t) - e_j\right] \quad j=1,\ldots,d, \quad \forall t \geq 0.
\end{align*}
Then from Lemma~\ref{lem:k}, for each $j \in \{1 \, ,\ldots, \, d\}$ and for $t \geq 0$ we have
\begin{align*}
    \norm{k^0_j(t)-k_{j\beta}} & \leq \rho \norm{k_j(t-1)-k_{j\beta}}. \nonumber
\end{align*}
Substituting from the definition~\eqref{eqn:noisy_k} we get
\begin{align*}
    \norm{k^0_j(t)-k_{j\beta}} & \leq \rho \norm{k^o_j(t-1) + w^k_j(t) - k_{j\beta}} \nonumber \\
    & \leq \rho \norm{ k^o_j(t-1)- k_{j\beta}} +\rho \norm{w^k_j(t)}. \nonumber
\end{align*}
Recall that $\Tilde{k}^0_j(t) = k^0_j(t) - k_{j\beta}$ for $j \in \{1 \, ,\ldots, \, d\}$ and for $t \geq 0$. Then the above inequality can be written as
\begin{align}
    \norm{\Tilde{k}^0_j(t)} & \leq \rho\norm{\Tilde{k}^0_j(t-1)} + \rho \norm{w^k_j(t)}, \, \forall t.
\end{align}
Using the upper bound from~\eqref{eqn:noise_bd} in above, we get
\begin{align}
    \norm{\Tilde{k}^0_j(t)} & \leq \rho\norm{\Tilde{k}^0_j(t-1)} + \rho ~ w, \, \forall t. \label{eqn:k_tilde1}
\end{align}
Using the recursion~\eqref{eqn:k_tilde1} $t$ times we get
\begin{align}
    & \norm{\Tilde{k}^0_j(t)} \leq \rho^{t} \norm{\Tilde{k}^0_j(0)} + \left(\rho+...+\rho^{t} \right)w. \label{eqn:k_tilde2}
\end{align}    
Using triangle inequality on~\eqref{eqn:noisy_k} we get
\begin{align}
    \norm{\Tilde{k}_j(t)} \leq  \norm{\Tilde{k}^0_j(t)} + \norm{w^k_j(t)} \overset{\eqref{eqn:noise_bd}}{\leq} \norm{\Tilde{k}^0_j(t)} + w. \label{eqn:k_tilde3}
\end{align}
Upon substituting from~\eqref{eqn:k_tilde2} in~\eqref{eqn:k_tilde3} we have
\begin{align}  
    \norm{\Tilde{k}_j(t)} \leq \rho^{t} \norm{\Tilde{k}^0_j(0)} + \left(1+\rho+...+\rho^{t} \right)w. \label{eqn:kcol_3}
\end{align}
Due to noise,~\eqref{eqn:x_update} which is equivalent to~\eqref{eqn:z_mltp} becomes
\begin{align}
    z^o(t+1) & = \left(I- \delta K(t+1)\A\right) z(t), \quad \forall t \geq 0. \label{eqn:znot_1}
\end{align}
Upon substituting from~\eqref{eqn:tilde_k} and~\eqref{eqn:noisy_z} in~\eqref{eqn:znot_1} we obtain that
\begin{align}
    z^o(t+1) & = \left(I- \delta K_{\beta}\A\right) \left( z^o(t)+w^x(t) \right) - \delta \widetilde{K}(t+1) \A\, \left( z^o(t)+w^x(t) \right). \label{eqn:x1}
\end{align}
From~\eqref{eqn:norm_frob_K} and~\eqref{eqn:matrix_norms} we get
\begin{align}
    \norm{\widetilde{K}(t)}^2 \leq \norm{\widetilde{K}(t)}^2_F = \sum_{j=1}^d \norm{\widetilde{k}_j(t)}^2. \nonumber
\end{align}
Substituting from~\eqref{eqn:kcol_3} in above we have
\begin{align}
    \norm{\widetilde{K}(t)}^2 \leq \underbrace{\sum_{j=1}^d \left( \rho^{t} \norm{\Tilde{k}^0_j(0)} + \left(1+\rho+...+\rho^{t} \right)w \right)^2}_{S(t)^2}
    \implies \norm{\widetilde{K}(t)} \leq S(t). \label{eqn:kt_st}
\end{align}
Applying the triangle inequality on~\eqref{eqn:x1} we get
\begin{align*}
    \norm{z^o(t+1)} & \leq \norm{I- \delta K_{\beta}\A} \left(\norm{z^o(t)}+\norm{w^x(t)} \right) + \delta \norm{\widetilde{K}(t+1)} \norm{\A} \left(\norm{z^o(t)}+\norm{w^x(t)} \right).
\end{align*}
Substituting from the bounds~\eqref{eqn:noise_bd} and~\eqref{eqn:kt_st} above we get
\begin{align}
    \norm{z^o(t+1)} & \leq \norm{I- \delta K_{\beta}\A} \left(\norm{z^o(t)}+w \right) + \delta S(t+1) \norm{\A} \left(\norm{z^o(t)}+w \right). \label{eqn:znot_2}
\end{align}
Since the solution set $X^*$ is singleton, from the argument in Appendix~\ref{prf:super_conv} we have the matrix $\A$ to be positive definite and $r=d$. Hence, the projection matrix $Q$ onto $\N(\A)^{\perp}$ is the identity matrix by definition (see~\eqref{eqn:def_Q}). In that case, from~\eqref{eqn:bnd_first_2} we have
\[ \norm{I- \delta K_{\beta}\A} \leq \mu, \]
for $0 < \delta < 2\left(1 + \frac{\beta}{\lambda_1}\right)$ and $\mu^* \leq \mu$. 
Since $X^*$ is singleton, the conditions of Corollary~\ref{cor:super_conv} hold. Thus, 
$\mu^* = 0$ for $\delta = 1$. Hence, for $\delta = 1$ we have
\[ \norm{I- \delta K_{\beta}\A} = 0. \]
Plugging the above equation and~\eqref{eqn:aQ_lambda} into~\eqref{eqn:znot_2} we have
\begin{align*}
    \norm{z^o(t+1)} \leq \underbrace{S(t+1) \lambda_1}_{R(t+1)} \left(\norm{z^o(t)}+w \right), \quad \forall t\geq 0.
\end{align*}
Using this recursion $t$ times we obtain that
\begin{align}
    & \norm{z^o(t+1)} \leq R(t+1)...R(1) \norm{z^o(0)} + \left(R(t+1) + R(t+1) R(t) + ... + R(t+1)...R(1) \right)w. \label{eqn:znot_3}
\end{align}
Again using triangle inequality on~\eqref{eqn:noisy_z} we get
\begin{align}
    \norm{z(t)} \leq  \norm{z^0_j(t)} + \norm{w^x(t)} \overset{\eqref{eqn:noise_bd}}{\leq} \norm{z^0(t)} + w. \label{eqn:znot_4}
\end{align}
Upon substituting from~\eqref{eqn:znot_3} in~\eqref{eqn:znot_4} we get
\begin{align}
    \norm{z(t+1)} \leq R(t+1)...R(1) \norm{z^o(0)} + \left(1+ R(t+1) + R(t+1) R(t) + ... + R(t+1)...R(1) \right)w. \label{eqn:x2}
\end{align}

{\bf Step II:} In this step, we obtain an upper bound on both the terms on the R.H.S. of~\eqref{eqn:x2}. For this step we define a notation \[ p_j(t) = \left( \rho^{t} \norm{\Tilde{k}^0_j(0)} + \left(1+\rho+...+\rho^{t} \right)w \right), \quad \forall t\geq 0 \]
so that $S(t)^2 = \sum_{j=1}^d p_j(t)^2$. \\

Then
\begin{align*}
    p_j(t) - p_j(t-1) & = \left(\rho^{t} - \rho^{t-1}\right) \norm{\Tilde{k}^0_j(0)} + \rho^{t} w 
    =  \rho^{t-1} \left(\rho \left(w+\norm{\Tilde{k}^0_j(0)}\right) - \norm{\Tilde{k}^0_j(0)}\right).
\end{align*}
Recall the definition of $\rho_j$:
\[\rho_j = \dfrac{\norm{\Tilde{k}_j^o(0)}}{\norm{\Tilde{k}_j^o(0)}+w}, \, j=1,...,d. \]
Using the condition~\eqref{eqn:assump_2a} in the above definition, we have that
\begin{align*}
    \rho \left(w+\norm{\Tilde{k}^0_j(0)}\right) < \norm{\Tilde{k}^0_j(0)}, \, j=1,...,d.
\end{align*}
This leads us to
$$p_j(t) < p_j(t-1) \, \forall t, \, j = 1,...,d.$$ 
Since $S(t)^2 = \sum_{j=1}^d p_j(t)^2$,
we have $S(t) < S(t-1) \, \forall t$. 
Similarly, as $R(t) = \lambda_1 S(t)$, we further have $R(t) < R(t-1) \, \forall t$. 
As $\rho < 1$, from the definition of $S(t)$ we obtain that \begin{align}
    & \lim_{t \rightarrow \infty} S(t)^2 = \sum_{j=1}^d \left(\dfrac{1}{1-\rho}w \right)^2 = d \, \left(\dfrac{w}{1-\rho} \right)^2 \nonumber\\
    \implies & \lim_{t \rightarrow \infty} R(t) =  \dfrac{ \lambda_1 \sqrt{d} w}{1-\rho}.
\end{align}
Observing that $\dfrac{1-\rho}{ \lambda_1 \sqrt{d}} = w_{bd}$ (see Appendix~\ref{sec:noise}), we get
\begin{align}
    & \lim_{t \rightarrow \infty} R(t) =  \dfrac{w}{w_{bd}}\overset{\eqref{eqn:assump_2b}}{<} 1. \label{eqn:r_lim}
\end{align}
Since $\lim_{t \rightarrow \infty} R(t) < 1$ and $0< R(t) < R(t-1) \, \forall t$, from Fact~\ref{fct:seq} there exists $T' < \infty$ such that $R(t) < 1 \, \forall t > T'$.  \\

Next, we upper bound the first term in~\eqref{eqn:x2}.
Using the fact that $\{R(t)\}$ is a strictly decreasing sequence, for $t > T'$ we get
\begin{align}
    & \Pi_{k=1}^t R(k) = \Pi_{k=1}^{T'} R(k) \, \Pi_{k=T'+1}^t R(k) < \Pi_{k=1}^{T'} R(1) \, \Pi_{k=T'+1}^t R(T'+1) = \left(R(1)\right)^{T'} \, \left(R(T'+1)\right)^{t-T'}. \nonumber
\end{align}
Since $R(T'+1) < 1$ and $R(1)$ is constant, the above inequality implies that
\begin{align}
    \lim_{t \rightarrow \infty} \Pi_{k=1}^t R(k) = 0. \label{eqn:x3}
\end{align}
~\\

Now we bound the second term in~\eqref{eqn:x2}. As $\{R(t)\}$ is strictly decreasing sequence, for $t > T'$ we have
\begin{align*}
   & 1 + R(t) + R(t)R(t-1) + \ldots + R(t)\ldots R(T'+1) + R(t)\ldots R(T'+1)R(T') + \ldots R(t)\ldots R(1) \\
   < & 1 + R(T'+1) + \left(R(T'+1)\right)^{2} + \ldots + \left(R(T'+1)\right)^{t-T'} + \left(R(T'+1)\right)^{t-T'}R(T') + \ldots + \\ & \left(R(T'+1)\right)^{t-T'} \, R(T')...R(1) \\
   = & 1 + R(T'+1) + \left(R(T'+1)\right)^{2} + \ldots + \left(R(T'+1)\right)^{t-T'} + \left(R(T'+1)\right)^{t-T'}\,\underbrace{\left(R(T')+\ldots + R(T')...R(1)\right)}_{constant}.
\end{align*}
Since $R(T'+1) < 1$, the above inequality implies that
\begin{align}
    \lim_{t \rightarrow \infty} \left(1 + R(t) + R(t) R(t-1) + ... + R(t)...R(1) \right) < \dfrac{1}{1 - R(T'+1)}. \label{eqn:x4}
\end{align}

The first statement in~\eqref{eqn:sse} follows from plugging the upper bounds~\eqref{eqn:x3} and~\eqref{eqn:x4} into~\eqref{eqn:x2}. The second statement in~\eqref{eqn:sse_lim} follows by taking limit and plugging in from~\eqref{eqn:r_lim}.
\end{proof}



\addtolength{\textheight}{-12cm}

\end{document}